\documentclass[12pt]{article}

\usepackage{amsmath, amsfonts, amsthm, graphicx}

\newcommand{\om}{\omega}

\newcommand{\halv}{\frac{1}{2}}

\newcommand{\calo}{{\cal O}}

\newcommand{\Rs}{\mbox{Re$(s)$}}

\newcommand{\lsls}{<\!<}
\newcommand{\bis}{^{\prime\prime}}

\newcommand{\Rk}{{\rm Re}\,(\kappa)}

\newcommand{\ub}{{\bf u}}

\newcommand{\p}{\partial}

\newtheorem{remark}{Remark}
\newtheorem{lemma}{Lemma}
\newtheorem{corollary}{Corollary}
\newtheorem{theorem}{Theorem}

\begin{document}

\title{Boundary estimates for the elastic wave equation in almost incompressible materials\thanks{This work was performed under the auspices of the U.S. Department
    of Energy by Lawrence Livermore National Laboratory under Contract DE-AC52-07NA27344.}}
% Normal mode analysis, influence of truncation errors, extreme material properties, ...
\author{Heinz-Otto Kreiss\thanks{Tr\"ask\"o-Stor\"o Institute of Mathematics, Stockholm, Sweden.}
  and N. Anders Petersson\thanks{Corresponding author. Center for Applied Scientific Computing, Lawrence Livermore
    National Laboratory, P.O. Box 808, Livermore, CA 94551, E-mail: {\tt andersp@llnl.gov}.}  }
\date{\today} \maketitle
\begin{abstract}
We study the half-plane problem for the elastic wave equation subject to a free surface boundary
condition, with particular emphasis on almost incompressible materials. A normal mode analysis is
developed to estimate the solution in terms of the boundary data, showing that the problem is {\it
  boundary stable}. The dependence on the material properties, which is difficult to analyze by the
energy method, is made transparent by our estimates. The normal mode technique is used to analyze
the influence of truncation errors in a finite difference approximation. Our analysis explains why
the number of grid points per wave length must be increased when the shear modulus ($\mu$) becomes
small, that is, for almost incompressible materials. To obtain a fixed error in the phase velocity
of Rayleigh surface waves as $\mu\to 0$, our analysis predicts that the grid size must be
proportional to $\mu^{1/2}$ for a second order method. For a fourth order method, the grid size can
be proportional to $\mu^{1/4}$. Numerical experiments confirm these scalings and illustrate the
superior efficiency of the fourth order method.

%Our analysis can be used to explain many of the numerical difficulties when solving the elastic wave
%equation for almost incompressible materials.

%Numerical examples show that a fourth order accurate
%method is more efficient than a second order method, and the advantage for the fourth order method
%becomes more pronounced as $\mu\to 0$.

\end{abstract}

\section{Introduction}

Consider the half-plane problem for the two-dimensional elastic wave equation in a homogeneous
isotropic material. By scaling time to give unit density, the displacement with Cartesian components $(u,v)^T$
is governed by
\begin{equation}\label{uv-eqn}
\left\{\begin{array}{r@{\,=\,}l}
 u_{tt}&\mu \Delta u + (\lambda + \mu)(u_{x}+v_{y})_x + F_1(x,y,t),\\
 v_{tt}&\mu \Delta v + (\lambda + \mu)(u_{x}+v_{y})_y + F_2(x,y,t),
\end{array}\right.
\quad x\geq 0,\ -\infty<y<\infty,\ t\geq 0,
\end{equation}
where $(F_1, F_2)^T$ is the internal forcing. Here, $\lambda$ and $\mu>0$ are the first and second
Lam\'e parameters of the material. We assume that both parameters are constant and $\lambda>0$. The
displacement is subject to initial conditions
\begin{equation}\label{uv-initial}
\begin{cases}u(x,y,0) = f_{10}(x,y),\\
u_t(x,y,0) = f_{20}(x,y),
\end{cases}\quad
\begin{cases}v(x,y,0) = f_{11}(x,y),\\
v_t(x,y,0) = f_{21}(x,y),
\end{cases}
\quad x\geq 0,\ -\infty<y<\infty.
\end{equation}
In this paper we consider normal stress boundary conditions along the $x=0$ boundary,
\begin{equation}
\begin{cases}
u_x+\gamma^2 v_y=g_1(y,t),\\
u_y+v_x=g_2(y,t),
\end{cases}\quad x=0,\ -\infty < y < \infty,\ t\geq 0,\label{uv-bc}
\end{equation}
where $g_1$ and $g_2$ are boundary forcing functions, and
\[
\gamma^2=\frac{\lambda}{2\mu+\lambda}.
\]
When $g_1=0$ and $g_2=0$, \eqref{uv-bc} is called a free surface boundary condition.

Since time was scaled to give unit density, the elastic energy is given by
\begin{equation}\label{elastic-energy}
E(t) = \frac{1}{2} \int_{-\infty}^\infty \int_{0}^\infty (u_t^2 + v_t^2) + \lambda(u_x + v_y)^2 + \mu \left(
2u_x^2 + 2 v_y^2 + (u_y + v_x)^2 \right)\, dx\, dy.
\end{equation}
It is well known (see e.g.~Achenbach~\cite{achenbach}, pp.~59-61) that the elastic energy satisfies
\begin{multline*}
\frac{d}{dt}E(t) = \int_{-\infty}^\infty\int_{0}^\infty (u_t F_1 + v_t F_2)\, dx dy \\
- \int_{-\infty}^\infty \left.\left(u_t
\left((2\mu+\lambda)u_x +\lambda v_y\right) + v_t \mu (u_y+v_x)\right)\right|_{x=0}\,dy.
\end{multline*}
In particular, without boundary  and interior forcing, the
elastic energy is conserved,
\begin{equation}\label{eq:energy6}
E(t) = E(0),\quad t>0, \quad  g_1=g_2=0,\ F_1=F_2=0.
\end{equation} 
Note that the elastic energy is a semi-norm of the solution. The energy estimate bounds this
semi-norm in terms of the initial data and the internal forcing $(F_1, F_2)^T$. For this reason, the
elastic wave equation is a well-posed problem. However, the energy estimate does not provide
detailed insight into how the solution depends on the material parameters, or the boundary data.

The material parameters, in particular the ratio $\mu/\lambda$, strongly influences the accuracy of
numerical solutions of the elastic wave equation. As a motivating example, we propagate a Rayleigh
surface wave using a second order accurate finite difference method. In the numerical experiment, we make the
$y$-direction 1-periodic and take the wave length to be one. A free surface boundary condition is
imposed at $x=0$. The Rayleigh surface wave propagates harmonically in the $y$-direction and decays
exponentially in $x$, see Figure~\ref{fig:7}. 
\begin{figure}
\begin{centering}
\includegraphics[width=0.45\textwidth]{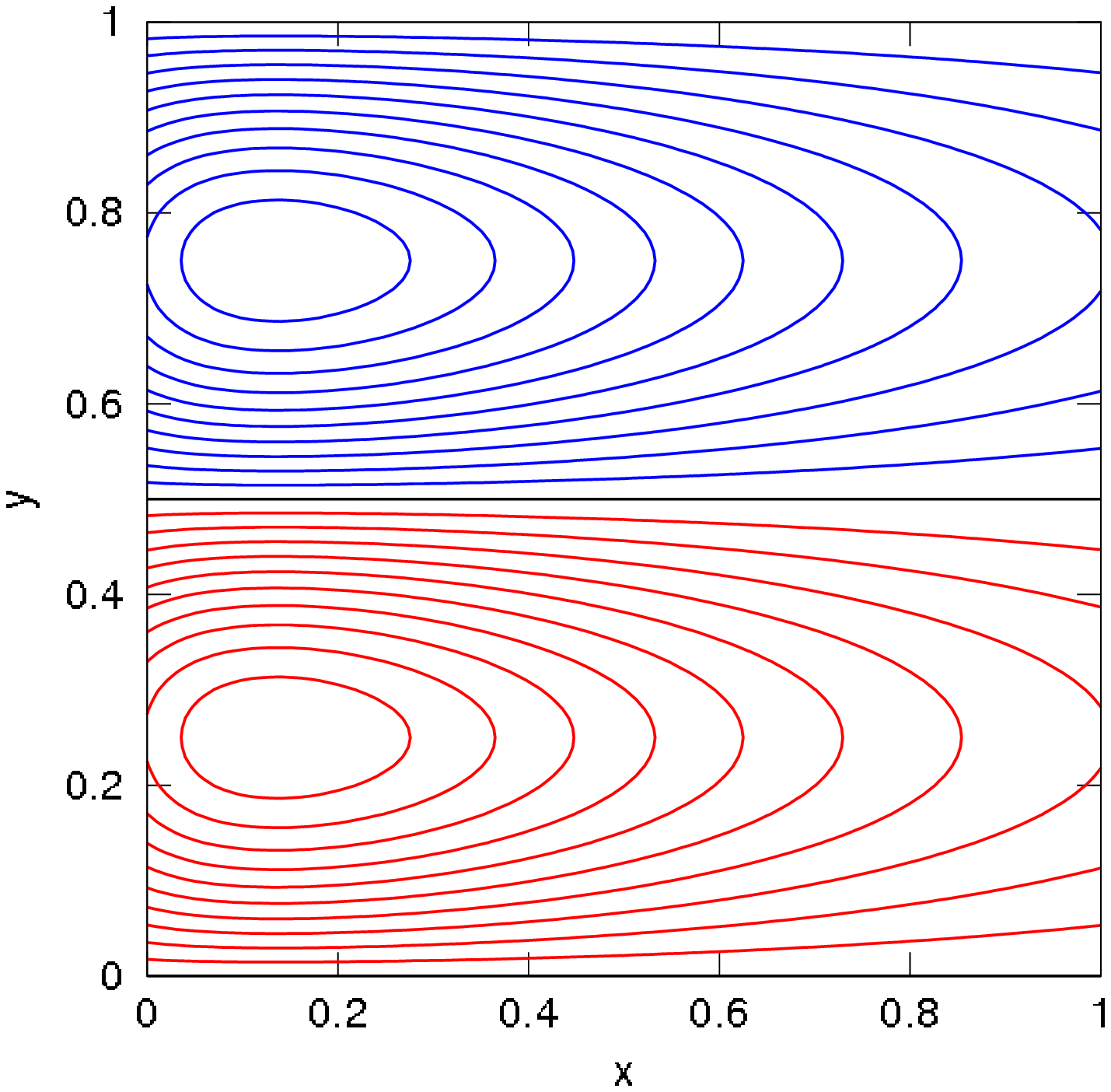}
\includegraphics[width=0.45\textwidth]{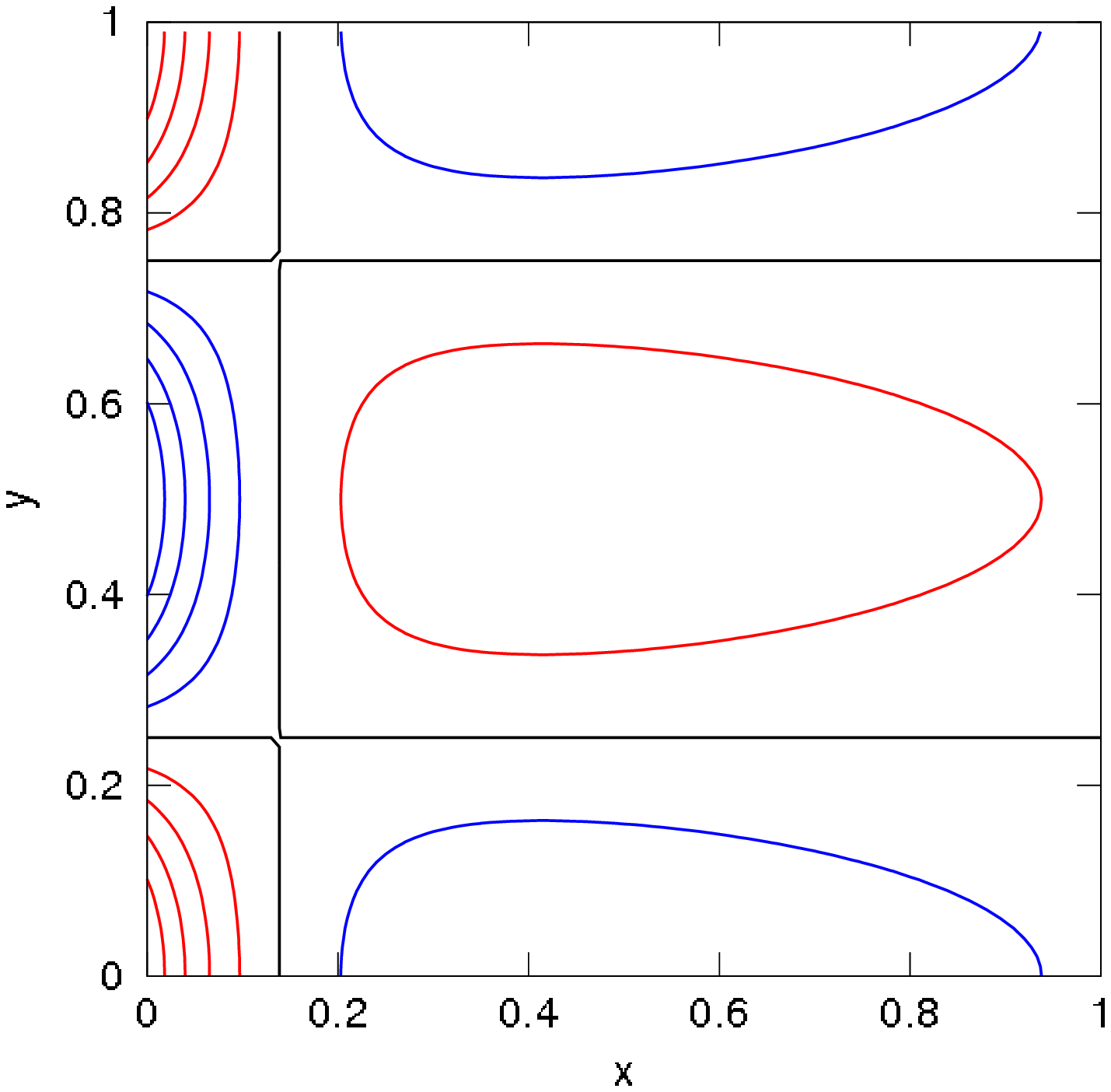}
\caption{A contour plot of a Rayleigh surface wave as function of $(x,y)$ at $t=0$ for a material
  with $\lambda=1$ and $\mu=0.01$. The $u$-component is shown to the left and the $v$-component to
  the right. The contour levels are given between -0.5 and 0.5, with spacing 0.05. Red
  and blue lines correspond to negative and positive values, respectively. The zero level is plotted
  in black. }
\label{fig:7}
\end{centering}
\end{figure}
We take $\lambda=1$ and vary $\mu$, which gives the surface wave a phase velocity that is
proportional to $\sqrt{\mu}$. We discretize the elastic wave equation on a grid with grid size $h$,
corresponding to $P=1/h$ grid points per wave length; further details of this numerical experiment
are presented in \S~\ref{sec:num}. In Figure~\ref{fig:max-err}, we report the error in the numerical
solution at time $t=20$. For the smaller values of $\mu$, a large number of grid points per wave length are needed to
obtain an acceptable error level and a second order convergence rate. For the finest mesh with $200$
grid points per wave length, the error increases by more than an order of magnitude (from $3.76\cdot
10^{-3}$ to $5.09\cdot 10^{-2}$), when $\mu$ decreases by two orders of magnitude (from $10^{-1}$ to
$10^{-3}$). Note that the gradient of the exact solution only depends weakly on $\mu$ and is of the
order ${\cal O}(1/2\pi)$ for all values of $\mu>0$. Hence, the loss of accuracy is not due to poor
resolution in space. Furthermore, the phase velocity of the surface wave becomes slower and slower
as $\mu\to 0$, while the time step is governed by $\sqrt{\lambda+3\mu}$, which tends to
$\sqrt{\lambda}=1$. Hence, the temporal resolution of the surface wave only improves as $\mu\to 0$.
\begin{figure}
 \begin{center}
% \hspace{25mm}   
\includegraphics[width=0.65\textwidth]{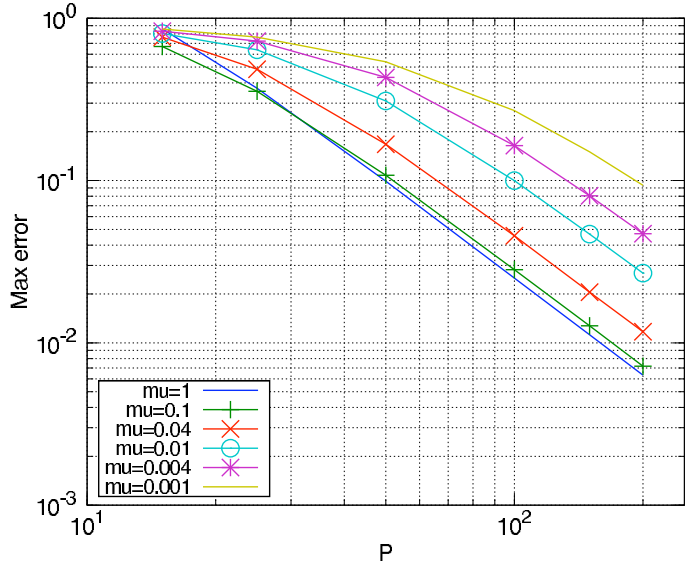}
    \caption{Max error, normalized by the max norm of the exact solution, at time $t=20$ in the
      numerical solution of the Rayleigh surface wave problem when $\lambda=1$ and different values
      of $\mu$. The error is shown as function of the number of grid points per wave length, $P=1/h$.}
    \label{fig:max-err}
 \end{center}
\end{figure}
%% Also note that the coarsest mesh has 15 grid points per wave length, corresponding to a resolution
%% that often is used in practical simulations of seismic wave propagation. As we can see, this
%% resolution results in a relative error of 60-80 \%.

In this paper we use a normal mode analysis to explain the loss of accuracy as $\mu\to 0$, which
corresponds to the incompressible limit of an elastic material. The normal mode analysis allows us
to estimate the solution in terms of the boundary data, and makes the dependence on the material
parameters transparent. We show that the solution is {\em strongly boundary stable}, except in the
vicinity of the generalized eigenvalues corresponding to surface waves. Here the solution is as
smooth as the boundary data, i.e., only {\em boundary stable} (see~\cite{Kreiss-Ortiz-Petersson} for
definitions of these stability concepts). We develop a modified equation model of the truncation
errors in the numerical calculation, where we view the discretized boundary conditions as a
perturbation of the exact boundary conditions. This analysis reveals how perturbations of the
boundary conditions influence the solution, and how the material parameters enter in the relation.

To analyze the solution of \eqref{uv-eqn}-\eqref{uv-bc}, we follow the technique used by
Kreiss, Ortiz and Petersson~\cite{Kreiss-Ortiz-Petersson} and split the problem into two
parts. First we consider a Cauchy problem, where the definition of the forcing and the initial data
are extended to the whole of ${\mathbb R}^2(x)$. Secondly, we subtract this solution from the
solution of the half-plane problem to obtain a new half-plane problem, where only the boundary data do
not vanish. This is a very natural procedure because all the difficulties and many physical
phenomena arise at the boundary. The new half-plane problem is analyzed in detail using the
Fourier-Laplace transform method, leading to estimates of the solution in terms of the boundary data.

The remainder of the paper is organized in the following way. The properties of the Cauchy problem
are briefly discussed in Section~\ref{sec:cauchy}. The normal mode analysis of the half-plane
problem is developed in Section~\ref{sec:half-plane}. We discuss the eigenvalue problem in
Section~\ref{sec:necessary}-\ref{sec:gen-eig}, leading to necessary conditions for a well-posed
problem. Boundary estimates are derived in Section~\ref{sec:b-force}.  In Section~\ref{sec:trunc},
we use the normal mode theory to perform a modified equation analysis of the discretized boundary
conditions. This analysis shows how the number of grid points per wave length must be increased to
maintain a given error level in the numerical solution when $\mu\to 0$. For the second order method,
the grid size must be proportional to $\mu^{1/2}$, while it suffices to take $h\sim \mu^{1/4}$ for
the fourth order method. These scalings are confirmed by the numerical experiments in
Section~\ref{sec:num}, illustrating that the fourth order method is significantly more efficient
than the second order approach, in particular for small values of $\mu$. Conclusions are given in
Section~\ref{sec:conc}.

\section{The Cauchy problem} \label{sec:cauchy}

In this section we consider the Cauchy problem for \eqref{uv-eqn}-\eqref{uv-initial}.  The
definitions of the forcing functions and the initial data can be smoothly extended to the
whole of ${\mathbb R}^2(x)$. For simplicity we use the same symbols for the extended functions as
for the original ones.

We start by deriving an equation for the divergence of the displacement, $\delta =
u_x + v_y$, by forming the divergence of \eqref{uv-eqn}. This gives
\begin{equation}\label{div-eqn}
\delta_{tt} = (\lambda+2\mu) \Delta \delta + G(x,y,t),\quad -\infty
< (x,y) < \infty,\ t\geq 0,
\end{equation}
where the forcing is $G=\p F_1/\p x + \p F_2/\p y$. The divergence, $\delta=\delta(x,y,t)$, is
subject to initial conditions
\begin{equation}\label{div-ini}
\delta(x,y,0) = \frac{\p f_{10}}{\p x} + \frac{\p f_{11}}{\p y},\quad 
\delta_t(x,y,0) = \frac{\p f_{20}}{\p x} + \frac{\p f_{21}}{\p y},\quad -\infty < (x,y) < \infty.
\end{equation}

By first solving the wave equation for the divergence, we can (in principle) treat the divergence as a forcing in
the Cauchy problems for $u$ and $v$,
\begin{equation}\label{uv-div-eqn}
\left\{\begin{array}{r@{\,=\,}l}
u_{tt}&\mu \Delta u + \widetilde{F}_1(x,y,t),\\
v_{tt}&\mu \Delta v + \widetilde{F}_2(x,y,t),
\end{array}\right.
\quad -\infty < (x,y) < \infty,\ t\geq 0,
\end{equation}
where
\[
\widetilde{F}_1(x,y,t) = (\lambda + \mu) \delta_x + F_1(x,y,t),\quad
\widetilde{F}_2(x,y,t) = (\lambda + \mu) \delta_y + F_2(x,y,t).
\]
Since $\delta$, $u$, and $v$ all satisfy scalar wave equations, we conclude that the Cauchy problem for the
elastic wave equation is well-posed. 

Note that the wave propagation speed in the wave equation for the divergence is
$\sqrt{\lambda+2\mu}$. For $G=0$, \eqref{div-eqn} admits plane wave solutions of the type
\[
\delta(x,y,t) = e^{i\omega(x \pm c_p\,t)},\quad c_p=\sqrt{\lambda+2\mu}.
\]
Hence, a wave with angular frequency $\xi = \omega c_p$ has wave length
\begin{equation}\label{p-wave}
L_p=\frac{2\pi}{\omega} = \frac{2\pi c_p}{\xi} = \frac{2\pi\sqrt{\lambda+2\mu}}{\xi}.
\end{equation}
Note that $L_p$ stays bounded for $\lambda=\mbox{const.}$, $\mu\to 0$. By taking the curl of
\eqref{uv-eqn}, we can also derive a scalar wave equation for the curl of the displacement,
where the wave propagation speed is $\sqrt{\mu}$. Hence, the elastic wave equation also admits plane
waves with wave length
\[
L_s = \frac{2\pi\sqrt{\mu}}{\xi}.
\]
The length of these waves tend to zero as $\mu\to 0$.

\section{The half-plane problem}\label{sec:half-plane}

We are interested in solutions with bounded $L^2$-norm and therefore we assume
\begin{equation}\label{bounded-L2}
\int_{-\infty}^\infty \int_0^\infty \left(|u|^2 + |v|^2\right) \, dx\, dy = \| \ub\|^2 <
\infty,\quad\mbox{for every fixed $t$}.
\end{equation}

Throughout the remainder of the paper, $s=\eta+i\xi$ denotes a complex number where $\eta$, $\xi$ are real numbers.
As a preliminary, we define the branch cut of $\sqrt{a+i b}$ by
\[
-\pi< {\rm arg}\,( a+i b)\le \pi,\quad {\rm arg}\,\sqrt{ a + i b}=\halv\, {\rm arg}\, (a+ i b),
\]
where $a$ and $b$ are real numbers, 

\subsection{A necessary condition for well-posedness, the eigenvalue problem}\label{sec:necessary}

We start with a test to find a necessary condition such that the half-plane problem is well
posed. 
\begin{lemma}\label{lem:necessary}
Let $F_1=F_2=0$ and $g_1=g_2=0$. The problem \eqref{uv-eqn}-\eqref{uv-bc} is not well-posed if we can
find a non-trivial simple wave solution of the type
\begin{equation}
\label{eq:4}
\begin{array}{c}
u=U(x)e^{st+i\om y},\quad  v=V(x)e^{st+i\om y},\\
|U|_\infty< \infty,\quad |V|_\infty<\infty,\quad \Rs > 0,\ \mbox{$\omega$ real}.
\end{array}
\end{equation}
\end{lemma}
\begin{proof}
If we have found such a solution, then
\[
u_1=U(\alpha x)e^{s\alpha t+i\om\alpha y},\quad  v_1=V(\alpha x)e^{s\alpha t+i\om\alpha y},
\]
is also a solution for any $\alpha>0$. Since $\Rs>0$, we can find solutions that grow arbitrarily
fast in time. The problem is therefore not well-posed.
\end{proof}
We shall now discuss whether there are such solutions. Introducing (\ref{eq:4}) into (\ref{uv-eqn})
gives
\begin{equation}\label{eig-1}
\left\{
\begin{array}{r@{\,=\,}l}
(s^2+\mu \om^2)U - (2\mu+\lambda)U_{xx} - i(\lambda+\mu)\om V_x & 0,\\
(s^2+(2\mu+\lambda)\om^2)V -\mu V_{xx} - i(\lambda+\mu)\om U_x & 0,
\end{array}\right.
\quad x\geq 0.
\end{equation}
To derive boundary conditions for $U$ and $V$, we insert \eqref{eq:4} into \eqref{uv-bc},
\begin{equation}\label{eig-2}
\left\{
\begin{array}{r@{\,=\,}l}
U_x + i\gamma^2\om V & 0,\\
i\om U + V_x & 0,
\end{array}\right.
\quad x=0.
\end{equation}

Equation \eqref{eig-1} is a system of linear ordinary differential equations with constant
coefficients. It can be solved using the ansatz
\begin{equation}\label{ode-ansatz}
U(x) = u_0 e^{-\kappa x},\quad V(x) = v_0 e^{-\kappa x}.
\end{equation}
Inserting \eqref{ode-ansatz} into \eqref{eig-1} gives a linear system for $(u_0, v_0)^T$, which
can be written
\begin{alignat}{2}
\left(s^2+\mu \om^2-(2\mu+\lambda)\kappa^2\right)u_0 + i(\lambda+\mu)\om\kappa\, v_0 &= 0,\label{eq:5a}\\
i(\lambda+\mu)\om\kappa\, u_0 + \left(s^2+(2\mu+\lambda)\om^2-\mu\kappa^2\right)v_0 &= 0.\label{eq:5b}
\end{alignat}
Let
\[
\zeta=s^2+\mu\om^2-\mu\kappa^2.
\]
Then we can write (\ref{eq:5a})-(\ref{eq:5b}) as
\begin{alignat}{2}
(\zeta-(\lambda+\mu)\kappa^2)\,u_0 + i(\lambda+\mu)\om\kappa\, v_0 &= 0,\label{eq:6a}\\
i(\lambda+\mu)\om\kappa\, u_0+(\zeta+(\lambda+\mu)\om^2)\,v_0      &= 0.\label{eq:6b}
\end{alignat}
This system has a non-trivial solution if and only if its determinant is zero,
\[
\left[\zeta-(\lambda+\mu)\kappa^2\right]\left[\zeta+(\lambda+\mu)\om^2\right]+
(\lambda+\mu)^2\om^2\kappa^2=0.
\]
%%i.e.,
%%\[
%%\zeta^2+(\lambda+\mu)(\om^2-\kappa^2)\zeta=0.
%%\]
There are two possibilities. Either $\zeta=0$, or $\zeta+(\lambda+\mu)(\om^2-\kappa^2)=0$,
corresponding to
\[
\kappa=\pm\sqrt{\om^2+\frac{s^2}{\mu}},\quad\mbox{or,}\quad
\kappa=\pm\sqrt{\om^2+\frac{s^2}{(2\mu+\lambda)}}.
\]

In appendix~\ref{app:a} we shall prove that there is a constant $\delta>0$ such that
\[
{\rm Re}\left(\sqrt{\om^2+\frac{s^2}{\mu}}\,\right) \geq \delta \,{\rm Re}(s),\quad
{\rm Re}\left(\sqrt{\om^2+\frac{s^2}{\lambda+2\mu}}\,\right) \geq \delta \,{\rm Re}(s),\quad {\rm Re}(s)>0.
\]
Thus, for $\Rs>0$,  there are two solutions that have bounded $L^2$-norm:
\begin{equation}\label{eq:8a}
\begin{pmatrix}
U(x) \\ V(x)
\end{pmatrix} = e^{-\kappa_1 x}\ub_1 +
e^{-\kappa_2 x}\ub_2 ,\quad \ub_j = \begin{pmatrix}u_{0j}\\ v_{0j}\end{pmatrix},\quad j=1,2,
\end{equation}
with
\begin{equation}
\label{eq:8b}
\kappa_1=\sqrt{\om^2+\frac{s^2}{\mu}},\quad \kappa_2=\sqrt{\om^2+\frac{s^2}{\lambda+2\mu}}.
\end{equation}
It is convenient to calculate the eigenvectors by inserting ($\kappa_1, \ub_1)$ into (\ref{eq:6b})
and $(\kappa_2, \ub_2)$ into (\ref{eq:6a}),
\begin{alignat*}{2}
i(\lambda+\mu)\om\kappa_1 u_{01} + (\lambda+\mu)\om^2 v_{01}  &=0, \\
-(\lambda+\mu)\om^2u_{02} + i(\lambda+\mu)\om \kappa_2 v_{02} &=0.
\end{alignat*}
Therefore,
\[
 v_{01}=-\frac{i\kappa_1}{\om}u_{01},\quad
 v_{02}=-\frac{i\om}{\kappa_2}u_{02}.
\]
We summarize these results in the following lemma.
\begin{lemma}\label{lem:gen-sol}
Assume $\Rs > 0$ and $\om\ne 0$. Then $\kappa_1\ne \kappa_2$ and the general solution of the
ordinary differential equation (\ref{eig-1}) can be written as
\begin{equation}
\begin{pmatrix}
U(x) \\ V(x)
\end{pmatrix} =
u_{01}
\begin{pmatrix}
1 \\ -\dfrac{i\kappa_1}{\om}
\end{pmatrix}  e^{-\kappa_1 x} + 
u_{02}
\begin{pmatrix}
1 \\ -\dfrac{i\omega}{\kappa_2}
\end{pmatrix}e^{-\kappa_2 x},
\label{eq:10}
\end{equation}
where $\kappa_1$ and $\kappa_2$ are given by \eqref{eq:8b}.
\end{lemma}
\begin{remark} Inserting (\ref{eq:10}) into (\ref{eq:4}) shows that all simple wave solutions satisfy
\begin{alignat}{2}
v_x - u_y &= \frac{i s^2}{\mu\om}\, u_{01}\, e^{st+i\om y - \kappa_1 x},\label{rotation}\\
u_x + v_y &= -\frac{s^2}{(\lambda+2\mu)\kappa_2}\,u_{02}\,e^{st+i\om y - \kappa_2 x}.\label{divergence}
\end{alignat}
Hence, $u_{01}$ and $u_{02}$ are proportional to the curl and divergence of the solution, respectively.
\end{remark}

Introducing \eqref{eq:10} into the boundary conditions \eqref{eig-2} gives
\begin{alignat}{2}
(1-\gamma^2)\kappa_1\kappa_2 \,u_{01} + (\kappa_2^2 - \gamma^2\om^2)\, u_{02} &= 0,\vspace{2mm}\label{bc1-hom}\\
(\kappa_1^2 + \om^2)\, u_{01} + 2\om^2\, u_{02} &=0.\label{bc2-hom}
\end{alignat}
%% We multiply the first equation by $-\kappa_2$ and the second by $-i\om$, 
%% \begin{equation}
%% \begin{pmatrix}
%% (1-\gamma^2)\kappa_1\kappa_2 & \left(\kappa_2^2 - \gamma^2\om^2 \right) \\
%% \left(\om^2 + \kappa_1^2\right) & 2 \om^2
%% \end{pmatrix}
%% \begin{pmatrix}u_{01} \\ u_{02} \end{pmatrix} = {\boldsymbol 0}. 
%% \label{u01-hom}
%% \end{equation}
The linear system \eqref{bc1-hom}-\eqref{bc2-hom} has a non-trivial solution if and only if its
determinant is zero,
\begin{equation}\label{det0}
\Delta =: 2\om^2(1-\gamma^2)\kappa_1\kappa_2 - \left(\kappa_2^2-\gamma^2\om^2\right)\left(\kappa_1^2 + \om^2\right)=0.
\end{equation}
Since $1-\gamma^2=2\mu/(\lambda+2\mu)$, we can write \eqref{det0} in the form
\[
\mu(\lambda+2\mu)\Delta = 4\mu^2\om^4\left(
\sqrt{1+\frac{s^2}{\om^2\mu}}\, \sqrt{1+\frac{s^2}{\om^2(\lambda+2\mu)}} - \left(1+\frac{s^2}{2\mu\om^2}\right)^2
\right) =: 4\mu^2\om^4 \varphi(\tilde{s}),
\]
where
\begin{equation}\label{varphi}
\varphi(\tilde{s}) =: \sqrt{1+\tilde{s}^2}\sqrt{1+\frac{\mu\tilde{s}^2}{\lambda + 2\mu}} - \left(
1+\frac{\tilde{s}^2}{2}\right)^2,\quad \tilde{s} = \frac{s}{|\om|\sqrt{\mu}}.
\end{equation}
%% which gives
%% \[
%% \frac{\kappa_1}{|\om|}=\sqrt{1+\tilde{s}^2},
%% \quad \frac{\kappa_2}{|\om|}=\sqrt{1+\frac{\mu\tilde{s}^2}{2\mu+\lambda}},
%% \quad s^2 + 2\mu\om^2 = 2\mu |\om|^2 \left( 1 + \frac{\tilde{s}^2}{2}\right),
%% \]
Note that the zeros of the determinant (\ref{det0}) are the solutions of $\varphi(\tilde{s})=0$.
\begin{lemma}\label{lem:res.gt.0}
Assume $\om\ne 0$. The function $\varphi(\tilde{s})$ does not have any zeros for ${\rm
  Re}\,(\tilde{s}) > 0$.
\end{lemma}
\begin{proof}
Assume there was a solution of $\varphi(\tilde{s})=0$ with ${\rm Re}\,(\tilde{s}) > 0$. It would
correspond to a non-trivial solution $(u_{01}, u_{02})$ of \eqref{bc1-hom}-\eqref{bc2-hom}. There
would therefore be a simple wave solution (\ref{eq:4}) where $U(x)$ and $V(x)$ are given by
(\ref{eq:10}). This simple wave solution would have $\Rs = |\om|\sqrt{\mu}\,{\rm Re}(\tilde{s}) >
0$, and for this reason, its elastic energy (\ref{elastic-energy}) would grow exponentially in
time. However, this is contradicted by the energy estimate \eqref{eq:energy6}, which says that the
elastic energy must be constant in time. There can therefore be no simple wave solutions for $\Rs > 0$, and
the function $\varphi(\tilde{s})$ can not have any zeros for ${\rm Re}\,(\tilde{s})>0$. 
\end{proof}
As a consequence of this lemma,
\begin{theorem}
The elastic wave equation (\ref{uv-eqn})-(\ref{uv-bc}) with $F_1=F_2=0$ and $g_1=g_2=0$, has no
simple wave solutions of the type (\ref{eq:4}), other than the trivial solution $u=v=0$.
\end{theorem}

Because \eqref{eig-1}-\eqref{eig-2} define an eigenvalue problem, we can also phrase the theorem as
\begin{theorem}
The eigenvalue problem \eqref{eig-1}-\eqref{eig-2} has no eigenvalues with $\Rs > 0$.
\end{theorem}
%% Assume there is a solution of $\varphi(\tilde{s}^2)=0$, such that ${\rm
%%   Re}\,(\tilde{s})=\tilde{\eta}>0$. There is then a non-trivial solution $(u_{01},u_{02})^T$ of
%% \eqref{bc1-hom}-\eqref{bc2-hom}-\eqref{u02-hom} that corresponds to a non-trivial simple wave solution of the form
%% \eqref{eq:4} with $U(x)$ and $V(x)$ given by \eqref{eq:10}. In physical space there is therefore a function
%% \[
%% \begin{pmatrix}
%% u\\v
%% \end{pmatrix} = 
%% u_{01}
%% \begin{pmatrix}
%% 1 \\ -\dfrac{i\kappa_1}{\om}
%% \end{pmatrix}\, e^{st+i\om y-\kappa_1 x}  +
%% u_{02}
%% \begin{pmatrix} 1 \\ -\dfrac{i\om}{\kappa_2}
%% \end{pmatrix}\, e^{st+i\om y-\kappa_2 x} ,
%% \]
%% that solves \eqref{uv-eqn}-\eqref{uv-bc} with $F_1=F_2=0$ and $g_1=g_2=0$. Because ${\rm
%%   Re}\,(s)=|\om|\sqrt{\mu}\,\tilde{\eta} >0$, the elastic energy \eqref{elastic-energy} of this
%% function grows exponentially in time. However, this is contradicted by \eqref{eq:energy6}, which
%% says that the elastic energy is conserved. There can therefore not be any roots of
%% $\varphi(\tilde{s}^2)=0$ with ${\rm Re}\,(\tilde{s})>0$.
%\end{proof}

\subsection{Generalized eigenvalues}\label{sec:gen-eig}

We shall now calculate the generalized eigenvalues, i.e., roots of the determinant \eqref{varphi} in
the limit ${\rm Re}\,(\tilde{s})\to 0+$. We need to discuss $\tilde{s}=i\tilde{\xi},~\tilde{\xi}$
real, and the zeros are given by
\begin{equation}\label{phi-xi-tilde}
\varphi(i\tilde{\xi})=:
\sqrt{1-\tilde{\xi}^2}\cdot
\sqrt{1-\frac{\mu\tilde{\xi}^2}{2\mu+\lambda}}
- \left(1-\frac{\tilde{\xi}^2}{2}\right)^2=0,
\quad \tilde{s}^2=-\tilde{\xi}^2.
\end{equation}
We have
\begin{lemma} \label{lem:roots}
Equation (\ref{phi-xi-tilde}) has the solution $\tilde{\xi}=0$, and exactly two solutions $\tilde{s}_0
= \pm i\tilde{\xi}_0$ with $ 0< \tilde{\xi}_0 <1$. There are no solutions with $\tilde{\xi}^2\geq 1$.
\end{lemma}
\begin{proof}
Inserting $\tilde{\xi}=0$ into \eqref{phi-xi-tilde} shows that it is a solution. Clearly, there are no
solutions for $1\leq \tilde{\xi}^2 < (2\mu+\lambda)/\mu$ because the first square root is purely
imaginary and the second square root is real. Also, the second term in $\varphi$ is always real and
negative.  For $\tilde{\xi}^2\geq (2\mu+\lambda)/\mu$, both square roots are purely imaginary and their
product is real and negative. Hence both terms in $\varphi$ are real and negative. We conclude that
there are no solutions for $\tilde{\xi}^2\geq 1$.

To analyze $0<\tilde{\xi}^2 < 1$, we denote $\tilde{\lambda}=\lambda/\mu$ and observe that the
function
\[
\psi(\sigma)=(1-\sigma)\left(1-\frac{\sigma}{2+\tilde\lambda}\right)-
\left(1-\frac{\sigma}{2}\right)^4,\quad \sigma=\tilde{\xi}^2,
\]
has the same roots as $\varphi$. It has the properties
\begin{enumerate}
\item 
\[
\psi(0)=0,\quad \psi(1)=-\frac{1}{16} <0,
\]
\item 
\[
d\psi/d\sigma =:\psi'(\sigma)=
-\left(1+\frac{1}{ 2+\tilde\lambda}\right)+\frac{2\sigma}{2+\tilde\lambda}+
2\left(1-\frac{\sigma}{2}\right)^3,
\]
that is,
\[
\psi'(0)=\frac{1+\tilde\lambda}{2+\tilde\lambda}>0,\quad
\psi'(1)=-\frac{2+3\tilde\lambda}{8+4\tilde\lambda}<0.
\]
\item 
\[
\psi\bis (\sigma)=\frac{2}{2+\tilde\lambda}-
3\left(1-\frac{\sigma}{2}\right)^2,
\]
that is,
\[
\psi\bis(0) = -\frac{4+3\tilde\lambda}{2+\tilde\lambda} < 0,\quad
\psi\bis(\sigma)=0~{\rm for}~ 
\frac{\sigma}{2}=1\pm\sqrt{\frac{2}{3(2+\tilde\lambda)}}.
\]
\end{enumerate}
Thus $\psi''(\sigma)$ has at most one sign change in $0\leq\sigma\leq 1$.  Properties 1--3 show that
$\psi'$ has one sign change and the lemma follows.
\end{proof}
In Table \ref{tab:1} we have calculated the scaled generalized eigenvalues $\tilde{s}_0^2 =
-\tilde{\xi}_0^2$ for some values of $\lambda/\mu$. Note that all values remain bounded in the limit
$\lambda/\mu\to\infty$, i.e. $\mu\to 0$ when $\lambda=\mbox{const}$. 
\begin{table}
\[
\begin{tabular}{|c|c|c|c|c|}\hline
$\lambda/\mu$ & $\tilde{s}_0^2=-\tilde{\xi}_0^2$ & $\kappa_{10}/|\omega|$ & $\kappa_{20}/|\omega|$ &
  $|\varphi'(\tilde{s}_0)|$ \\ \hline\hline
0         & -0.7639 & 0.4858 & 0.7861 & 0.6036 \\
1         & -0.8452 & 0.3933 & 0.8474 & 1.0610 \\
4         & -0.8877 & 0.3350 & 0.9230 & 1.6045 \\
8         & -0.8991 & 0.3175 & 0.9539 & 1.8360 \\
$\infty$  & -0.9126 & 0.2955 & 1      & 2.1936 \\
\hline
\end{tabular}
\]
\caption{ Coefficients in the solution at the generalized eigenvalues $s_0=\pm i\sqrt{\mu}\,|\om|\tilde{\xi}_0$, for
  some values of $\lambda/\mu$. }\label{tab:1}
\end{table}

Differentiating \eqref{varphi} gives
\begin{equation}\label{pert-det}
\varphi'(\tilde{s}) = \frac{\tilde{s}}{\sqrt{1+\tilde{s}^2}}\, \sqrt{1+\frac{\mu\tilde{s}^2}{2\mu+\lambda}}
+  \frac{\tilde{s} \mu}{2\mu+\lambda}\frac{\sqrt{1+\tilde{s}^2}}{\sqrt{1+\frac{\mu\tilde{s}^2}{2\mu+\lambda}}} -
2\tilde{s}\left( 1+\frac{\tilde{s}^2}{2} \right).
\end{equation}
Because $\varphi(\tilde{s}_0)=0$, \eqref{varphi} gives
\begin{equation}\label{varphi-prime}
\sqrt{1+\tilde{s}_0^2}\sqrt{1+\frac{\mu\tilde{s}_0^2}{2\mu+\lambda}} \frac{\varphi'(\tilde{s}_0)}{\tilde{s}_0} =
1 + \frac{\mu\tilde{s}_0^2}{2\mu+\lambda} + \frac{\mu(1 + \tilde{s}_0^2)}{2\mu+\lambda} - 2\left(1+\frac{\tilde{s}_0^2}{2}\right)^3
:=C_0,
\end{equation}
where $C_0 = C_0(\lambda/\mu)$ is real. Since $\tilde{s}_0$ is purely imaginary,
$\varphi'(\tilde{s}_0)$ is also purely imaginary. We report numerical values of $|\varphi'(\tilde{s}_0)|$ in
Table~\ref{tab:1}, demonstrating that $\varphi'(\tilde{s}_0)$ is bounded away from zero for all
values of $\lambda/\mu\geq 0$. Therefore, $\varphi(\tilde{s})$ has a first order zero at the generalized
eigenvalues $\tilde{s}_0=\pm i\tilde{\xi}_0$.

To calculate the eigenfunctions corresponding to the generalized eigenvalues $s_0=\pm i\xi_0$, we
consider the two boundary conditions \eqref{eig-2}. Evaluating the general solution \eqref{eq:10} gives
\[
i\om U + V_x = i\omega \left(
\left(1+\frac{\kappa_1^2}{\omega^2}\right)u_{01}+2u_{02}\right),\quad  x=0.
\]
At the generalized eigenvalues,
\[
\kappa_1^2 = \omega^2(1+\tilde{s}_0^2) = \omega^2(1-\tilde{\xi}_0^2).
\]
Hence, $i\om U +V_x = 0$ if
\begin{equation}\label{coeffs}
(2-\tilde{\xi}_0^2)u_{01}+2u_{02} = 0.
\end{equation}
If relation \eqref{coeffs} is satisfied, also $U_x+i\gamma^2\om V =0$. The eigenfunction corresponding
to $s_0=\pm i\xi_0$ is therefore given by
\begin{equation}\label{gen-efcn}
\begin{pmatrix}
u\\v
\end{pmatrix} = e^{\pm i\xi_0 t+i\om y -\kappa_{10} x} 
\begin{pmatrix}1\\ -\dfrac{i\kappa_{10}}{\om}\end{pmatrix}+
\frac{1}{2}(\tilde{\xi}_0^2-2) e^{\pm i\xi_0 t+i\om y -\kappa_{20} x}
\begin{pmatrix}1\\ -\dfrac{i\om}{\kappa_{20}}\end{pmatrix},
\end{equation}
where
\[
\kappa_{10}=|\omega|\sqrt{1-\tilde{\xi}_0^2},\quad
\kappa_{20}=|\omega|\sqrt{1-\frac{\mu\tilde{\xi}_0^2}{\lambda + 2\mu}},\quad \tilde{\xi}_0 =
\frac{\xi_0}{|\omega|\sqrt{\mu}}. 
\]
These eigenfunctions, also known as Rayleigh waves (see e.g.~Achenbach~\cite{achenbach}, \S 5.11),
represent surface waves that propagate in the positive or negative $y$-direction.

Now we consider the potential generalized eigenvalue $\tilde{s}=0$. Relations \eqref{varphi} and
\eqref{pert-det} show that both $\varphi=0$ and $\p\varphi/\p\tilde{s}=0$ for
$\tilde{s}=0$. Differentiating \eqref{pert-det} shows that $\p^2\varphi/\p\tilde{s}^2 \ne 0$ for
$\tilde{s}=0$. Thus $\varphi(\tilde{s})$ has a zero of order two at $\tilde{s}=0$. However, for
$\tilde{s}\to 0$, \eqref{eq:8b} show that both $\kappa_1\to |\om|$ and $\kappa_2\to |\om|$. In this
limit, boundary conditions \eqref{bc1-hom} and \eqref{bc2-hom} give
\[
u_{01}+u_{02} = 0,\quad \tilde{s}\to 0.
\]
Expanding the general solution \eqref{eq:10} around $\tilde{s}=0$ shows that the eigenfunction
vanishes identically in this limit. Thus $\tilde{s}=0$ is not a generalized eigenvalue.

\subsection{Boundary forcing}\label{sec:b-force}

As we discussed in the introduction, we split the solution of the half-plane problem
\eqref{uv-eqn}-\eqref{uv-bc} into a Cauchy problem and a new half-plane problem, where only the
boundary data do not vanish. Hence, the Cauchy problem satisfies the initial
conditions and the interior forcing function. Its solution drives the solution of the new half-plane
problem through a modified boundary forcing function. For example, when the half-plane problem
\eqref{uv-eqn}-\eqref{uv-bc} has an interior forcing function with compact support in $\bar{\Omega}$,
the solution of the Cauchy problem consists of waves propagating outwards from $\bar{\Omega}$. The gradient
of these waves along $x=0$ enter in the boundary forcing functions for the new half-plane problem.

The estimates obtained in this and the following sections are expressed in Fourier-Laplace
transformed space. It is clear that all these estimates have their counterpart in physical space. To
understand the relation between both types of estimates, we refer to chapter 7.4
of~\cite{KreissLorenz} or chapter 10 of~\cite{Gustafsson-Kreiss-Oliger}.

We consider \eqref{uv-eqn}-\eqref{uv-bc} with homogeneous initial data and internal forcing,
$F_1=F_2=0$. We Laplace transform the problem with respect to $t$, Fourier transform it with respect
to $y$, and denote the dual variables by $s$ and $\omega$, respectively. Here $\omega$ is a real
number and $s$ is complex. We obtain,
\begin{equation}\label{uv-laplace}
\left\{\begin{array}{r@{\,=\,}l}
 (s^2 + \mu\om^2 ) \hat{u} - ( 2\mu + \lambda ) \hat{u}_{xx}  - i\om (\lambda + \mu) \hat{v}_x & 0,\\
 (s^2 + (2\mu+\lambda) \om^2 )\hat{v} - \mu \hat{v}_{xx} - i\om (\lambda + \mu) \hat{u}_{x} & 0,
\end{array}\right.
\quad x\geq 0,\ \Rs>0,
\end{equation}
subject to the boundary condition
\begin{equation}\label{uv-bc-laplace}
\begin{cases}
\hat{u}_x+i \gamma^2\om \hat{v}=\hat{g}_1(\omega,s),\\
i\om \hat{u}+\hat{v}_x=\hat{g}_2(\omega,s),
\end{cases}\quad x=0.
\end{equation}

Note that $(\hat{u}, \hat{v})^T$ satisfy the same differential equation as $(U, V)^T$ in \eqref{eig-1}.
By Lemma~\ref{lem:gen-sol}, the general solution is of the form \eqref{eq:10}, i.e.,
\begin{equation}\label{uv-hat}
\begin{pmatrix}
\hat{u}(x) \\ \hat{v}(x)
\end{pmatrix} = \hat{u}_{01}
\begin{pmatrix}
1 \\ -\dfrac{i\kappa_1}{\om}
\end{pmatrix}\, e^{-\kappa_1 x} + \hat{u}_{02}
\begin{pmatrix}
1 \\ -\dfrac{i\omega}{\kappa_2}
\end{pmatrix}\, e^{-\kappa_2 x}.
\end{equation}
In the following, we assume $\om\ne 0$. The case $\om\to 0$ will be studied separately in appendix~\ref{sec:omega->0}.

By inserting \eqref{uv-hat} into boundary condition \eqref{uv-bc-laplace}, we get
\begin{alignat*}{2}
\left(1-\gamma^2\right) \kappa_1 \kappa_2 \,\hat{u}_{01} + \left(\kappa_2^2 - \gamma^2 \om^2
\right)\, \hat{u}_{02} &=-\kappa_2\, \hat{g}_1,\vspace{2mm}\\
\left( \kappa_1^2 + \om^2 \right)\,\hat{u}_{01} + 2\om^2\, \hat{u}_{02} &= -i\om\,\hat{g}_2.
\end{alignat*}
This system corresponds to \eqref{bc1-hom}-\eqref{bc2-hom} with an inhomogeneous right hand
side. In terms of the scaled variable $\tilde{s}$ defined by \eqref{varphi},
\begin{equation}\label{scaled-kappa}
\kappa_1 = |\om|\sqrt{1+\tilde{s}^2},\quad \kappa_2 =
|\om|\sqrt{1+\frac{\mu\tilde{s}^2}{\lambda+2\mu}}.
\end{equation}
After some algebra, the system for $(\hat{u}_{01}, \hat{u}_{02})^T$ becomes
\begin{alignat}{2}
\sqrt{1+\tilde{s}^2}\sqrt{1+\frac{\tilde{s}^2\mu}{2\mu+\lambda}}\, \hat{u}_{01} + 
\left(1+\frac{\tilde{s}^2}{2}\right)\,\hat{u}_{02}&=
-\frac{(\lambda + 2\mu)\hat{g}_1}{2\mu|\omega|}\sqrt{1+\frac{\tilde{s}^2\mu}{2\mu+\lambda}},\label{eq:21}\\
\left( 1 + \frac{\tilde{s}^2}{2}\right) \, \hat{u}_{01} + \hat{u}_{02} &= -\frac{i\, \hat g_2}{2\omega}.\label{eq:22}
\end{alignat}
The determinant of \eqref{eq:21}-\eqref{eq:22} is
\[
\varphi(\tilde{s}) =\sqrt{1+\tilde{s}^2}\sqrt{1+\frac{\tilde{s}^2\mu}{2\mu+\lambda}} -
\left(1+\frac{\tilde{s}^2}{2}\right)^2,
\]
where the function $\varphi(\tilde{s})$ was previously defined by \eqref{varphi}. To solve the
system, we eliminate $\hat{u}_{02}$ from \eqref{eq:22} and insert in \eqref{eq:21},
\begin{equation}\label{u01-g}
\varphi(\tilde{s})\, \hat{u}_{01} =
-\frac{(\lambda + 2\mu)\hat{g}_1}{2\mu|\omega|}\sqrt{1+\frac{\tilde{s}^2\mu}{2\mu+\lambda}} +
\frac{i\, \hat g_2}{2\omega} \left(1+\frac{\tilde{s}^2}{2}\right).
\end{equation}
Inserting this expression into \eqref{eq:22} gives
\begin{equation}
\varphi(\tilde{s})\,\hat{u}_{02} = \frac{(\lambda+2\mu)\, \hat{g}_1}{2\mu|\om|}
  \sqrt{1+\frac{\mu\tilde{s}^2}{\lambda+2\mu}}\left(1+\frac{\tilde{s}^2}{2}\right) -
  \frac{i\, \hat{g}_2}{2\om}\sqrt{1+\tilde{s}^2}\sqrt{1+\frac{\mu\tilde{s}^2}{\lambda+2\mu}}.
\label{u02-g}
\end{equation}
Hence, the system \eqref{eq:21}-\eqref{eq:22} becomes singular exactly at the roots of
$\varphi(\tilde{s})=0$. For $\Rs\geq 0$, Lemmas~\ref{lem:res.gt.0} and~\ref{lem:roots} prove that
this can only happen at the generalized eigenvalues. The general theory
of~\cite{Kreiss-Ortiz-Petersson} tells us that, away from the generalized eigenvalues,
$|\varphi(\tilde{s})|^{-1}$ is bounded and the problem is therefore strongly boundary stable.

We want to estimate the solution on the boundary in terms of the boundary forcing. For $x=0$, the general
solution \eqref{uv-hat} satisfies
\begin{equation}\label{uv-bndry}
\left\{\begin{array}{r@{\,=\,}l}
\hat{u}(0) & \hat{u}_{01} + \hat{u}_{02},\vspace{2mm}\\
\hat{v}(0) & -\dfrac{i\kappa_1}{\om} \hat{u}_{01} - \dfrac{i\om}{\kappa_2}\hat{u}_{02} =
-\dfrac{i|\om|}{\om}\sqrt{1+\tilde{s}^2}\,\hat{u}_{01}
-\dfrac{i\om}{|\om|}\left(1+\dfrac{\mu\tilde{s}^2}{\lambda+2\mu}\right)^{-1/2} \, \hat{u}_{02}, 
\end{array}\right.
\end{equation}

We now discuss how the solution behaves close to the generalized eigenvalues $\tilde{s}_0=\pm
i\tilde{\xi}_0$. By Lemma~\ref{lem:roots}, we have $0<\tilde{\xi}_0<1$ and both $\kappa_1$ and
$\kappa_2$ are real. Since $\varphi(\tilde{s}_0)=0$, Taylor expansion gives
\begin{equation}\label{taylor-gen-eig}
\varphi(\tilde{s}) = (\tilde{s} - \tilde{s}_0) \varphi'(\tilde{s}_0) +
\calo(|\tilde{s} - \tilde{s}_0|^2).
\end{equation}
Formula \eqref{varphi-prime} and Table~\ref{tab:1} shows that $|\varphi'(\tilde{s}_0)|\geq C_0\simeq
0.6$ for all $\lambda/\mu\geq 0$. We have $s-s_0 = \eta + i(\xi-\xi_0)$, and to leading order in
$0<\eta\ll 1$,
\begin{equation}\label{varphi-exp}
\varphi(\tilde{s}) =
\frac{s-s_0}{\sqrt{\mu}\,|\om|} \varphi'(\tilde{s}_0) + {\cal O}(\eta^2),
\end{equation}
which leads to the estimates
\begin{equation}\label{varphi-bound}
|s-s_0| \geq \eta,\quad |\varphi(\tilde{s})| \geq \frac{\eta}{\sqrt{\mu}|\om|} C_0,\quad \eta>0.
\end{equation}
For $\eta>0$, the system \eqref{eq:21}-\eqref{eq:22} is non-singular and we can substitute
\eqref{varphi-exp} into the solution formulas \eqref{u01-g}-\eqref{u02-g} to calculate $\hat{u}_{01}$
and $\hat{u}_{02}$. Inserting these values in \eqref{uv-bndry} and applying the triangle inequality
proves the following lemma.
\begin{lemma}
Let $s=\pm i\xi_0+s'$, $0<|s'|\ll 1$, where $\xi_0 = \sqrt{\mu}\,|\om|\tilde{\xi}_0$ and
$\varphi(\pm i\tilde{\xi}_0)=0$ with $0<\tilde{\xi}_0<1$. Also assume ${\rm Re}\,(s') = \eta >
0$. Then, the solution of \eqref{uv-laplace}-\eqref{uv-bc-laplace} satisfies the boundary estimate
\begin{alignat}{2}
|\hat{u}(0)| &\leq \frac{K}{\eta}
\left[
  \dfrac{2\mu+\lambda}{\sqrt{\mu}}\,|\hat{g}_1|
 + \sqrt{\mu}\, |\hat{g}_2| \right],\\
|\hat{v}(0)| &\leq \frac{K}{\eta}
\left[
\dfrac{2\mu+\lambda}{\sqrt{\mu}}\,|\hat{g}_1| +
\sqrt{\mu}\, |\hat{g}_2|\right],
\end{alignat}
where the constant $K>0$ is independent of $\mu$ and $\lambda$.  The solution is as smooth as the
boundary data and is therefore {\em boundary stable}. The solution operator has a simple pole at
$s_0=\pm i\xi_0$ and as a consequence, the solution in physical space grows linearly in time. The
growth rate is proportional to $|\hat{g}_1|/\sqrt{\mu}$ as $\mu\to 0$.
\end{lemma}

We shall now discuss the case $s\to 0$ in more detail. We assume $|\om|\geq \om_0>0$, which implies
$\tilde{s}\to 0$. Note that the eigenvectors in the general solution \eqref{uv-hat} become linearly
dependent in the limit, because both $\kappa_1= |\om|$ and $\kappa_2= |\om|$ for $\tilde{s}=0$. We
therefore assume ${\rm Re}\,\tilde{s} = \tilde{\eta} > 0$, and study the the solution in the limit
$|\tilde{s}|\to 0$.

Because $|\tilde{s}|\ll 1$, we can simplify (\ref{eq:21}) to
\begin{equation}\label{eq:21'}
\left(1+\frac{\tilde{s}^2}{2}\left(1+\frac{\mu}{2\mu+\lambda}\right)\right) \, \hat{u}_{01} + \left(
1+\frac{\tilde{s}^2}{2} \right)\, \hat{u}_{02} = 
-\frac{(2\mu+\lambda)\hat{g}_1}{2\mu|\omega|}\left(1+\frac{\tilde{s}^2}{2}\frac{\mu}{2\mu+\lambda}\right).
\end{equation}
We eliminate $\hat{u}_{02}$ using (\ref{eq:22}) and obtain
\begin{multline*}
\left(1+\frac{\tilde{s}^2}{2}\left(1+\frac{\mu}{2\mu+\lambda}\right)
- \left(1+\frac{\tilde{s}^2}{2}\right)^2 \right) \, \hat{u}_{01}
\\
=
-\frac{(2\mu+\lambda)\hat{g}_1}{2\mu|\omega|}\left(1+\frac{\tilde{s}^2}{2}\frac{\mu}{2\mu+\lambda}\right)
+ \frac{i\hat{g}_2}{2\omega} \left( 1 + \frac{\tilde{s}^2}{2} \right).
\end{multline*}
For small $|\tilde{s}|^2$ we obtain to first approximation
\begin{equation}\label{eq:24}
\tilde{s}^2 \hat{u}_{01} = \frac{(\lambda+2\mu)^2\, \hat{g}_1}{(\lambda+\mu)\mu |\om|} -
\frac{i(\lambda+2\mu)\,\hat{g}_2}{(\lambda+\mu)\om}. 
\end{equation}
Relation \eqref{eq:22} can be written
\begin{alignat}{2}
\hat{u}_{01} + \hat{u}_{02} &= -\frac{\tilde{s}^2}{2} \hat{u}_{01} - \frac{i\, \hat{g}_{2}}{2\om}
\nonumber \\
&=
-\frac{(\lambda+2\mu)^2\, \hat{g}_{1}}{2(\lambda+\mu)\mu|\om|} +
\frac{i\mu\,\hat{g}_2}{2(\lambda+\mu)\om}.\label{u02-hat-s0}
\end{alignat}
The solution on the boundary is given by \eqref{uv-bndry}. The first component satisfies
$\hat{u}(0)=\hat{u}_{01}+\hat{u}_{02}$, and \eqref{u02-hat-s0} shows that $\hat{u}(0)$ is bounded
independently of $\tilde{s}$. The expression for the second component can be simplified for
$|\tilde{s}|\ll 1$. We have to leading order
\begin{multline}\label{v-hat-s0}
\hat{v}(0) = -\frac{i|\om|}{\om}\left(1 + \frac{\tilde{s}^2}{2}\right)\,\hat{u}_{01}
-\frac{i|\om|}{\om}\left(1 - \frac{1}{2}\frac{\mu\tilde{s}^2}{\lambda+2\mu}\right)\,\hat{u}_{02}
\\ 
= -\frac{i|\om|}{\om}\left( \hat{u}_{01}+\hat{u}_{02}\right)
-\frac{i|\om|}{\om}\frac{\tilde{s}^2}{2}\left( \hat{u}_{01} -
\frac{\mu}{\lambda+2\mu}\hat{u}_{02}\right).
\end{multline}
Therefore, also $\hat{v}(0)$ is bounded independently of $\tilde{s}$. The factor $\om$ in the
denominator of the right hand side of \eqref{u02-hat-s0} gives the desired result that our problem
is {\em strongly boundary stable} at $\tilde{s}=0$.

\section{Influence of truncation errors on the generalized eigenvalues}\label{sec:trunc}

Consider the homogeneous differential equations \eqref{uv-eqn} with boundary conditions
\eqref{uv-bc}. Let
\begin{equation}\label{perturbed-bc}
g_1=\alpha_1h^2u_{xxx}+\alpha_2h^2v_{yyy},\quad 
g_2=\beta_1h^2v_{xxx}+\beta_2h^2u_{yyy},
\end{equation}
denote the principal part of the truncation error in a second order accurate method with grid size
$h$. We can think of boundary conditions \eqref{uv-bc} with boundary data \eqref{perturbed-bc}
as modified homogeneous boundary conditions. Again, we solve the problem using the technique of
Section \ref{sec:b-force} in terms of the simple wave ansatz \eqref{uv-hat}. By section
\ref{sec:b-force}, the modified boundary conditions become
\begin{alignat*}{2}
(1-\gamma^2)\kappa_1\kappa_2 \hat{u}_{01}+(\kappa_2^2-\gamma^2\om^2)\hat{u}_{02}+\kappa_2\hat g_1&=0,\\
(\om^2+\kappa_2^2)\hat{u}_{01} + 2\om^2 \hat{u}_{02} + i\om \hat g_2&=0,
\end{alignat*}
where 
\begin{alignat*}{2}
\hat g_1&=-\alpha_1h^2(\kappa_1^3 \hat{u}_{01} + \kappa_2^3 \hat{u}_{02}) - \alpha_2 h^2
\left(\om^2\kappa_1 \hat{u}_{01} + \frac{\om^4}{\kappa_2} \hat{u}_{02}\right),\\
\hat g_2&= i\beta_1h^2\left( \frac{\kappa_1^4}{\om}\hat{u}_{01} + \kappa_2^2\om \hat{u}_{02}\right) -
i\beta_2h^2\om^3(\hat{u}_{01}+ \hat{u}_{02}).
\end{alignat*}
For small $\mu$, the main effect comes from $\hat g_1$. For simplicity, we therefore assume that
$\hat g_2=0$ and obtain the equations
\begin{equation}\label{right-hand-side1}
\hat g_1 = -h^2(\alpha_1\kappa_1^3 + \alpha_2\om^2\kappa_1)\hat{u}_{01} -
h^2\left(\alpha_1\kappa_2^3 + \alpha_2\frac{\om^4}{\kappa_2}\right) \hat{u}_{02},\quad
\hat g_2=0.
\end{equation}
Introducing the scaled variable $\tilde{s}$ and the formulas for $\kappa_j$ according to
\eqref{scaled-kappa} gives us relations \eqref{eq:21}-\eqref{eq:22} with $\hat{g}_2=0$. By using the
homogeneous equation \eqref{eq:22}, we can eliminate $\hat{u}_{02}$ from \eqref{eq:21} and
\eqref{right-hand-side1}, resulting in the solution formula \eqref{u01-g} with $\hat{g}_2=0$, and
\[
\hat{g}_1 = -h^2(\alpha_1\kappa_1^3 + \alpha_2\om^2\kappa_1)\hat{u}_{01} +
h^2\left( 1+\frac{\tilde{s}^2}{2}\right)
\left(\alpha_1\kappa_2^3 + \alpha_2\frac{\om^4}{\kappa_2}\right) \hat{u}_{01}.
\]
Hence, the solution formula \eqref{u01-g} defines a perturbed eigenvalue problem that can be written in
the form 
\[
\varphi(\tilde{s})\hat{u}_{01} = \theta(\tilde{s}^2) \hat{u}_{01}.
%\quad \theta(\tilde{s}^2)
%\hat{u}_{01} = -\frac{(\lambda+2\mu)\hat{g}_1}{2\mu|\om|}
%\sqrt{1+\frac{\tilde{s}^2\mu}{\lambda+2\mu}}.
\]
Since 
\[ 
\frac{\kappa_2}{|\om|}=\sqrt{1+\frac{\tilde{s}^2\mu}{\lambda+2\mu}},
\]
we have
\[
\theta(\tilde{s}^2) = \frac{(\lambda+2\mu)h^2}{2\mu|\om|} \frac{\kappa_2}{|\om|}
\left(\alpha_1\kappa_1^3+\alpha_2\om^2\kappa_1- \left(1+\frac{\tilde
  s^2}{2}\right)\left(\alpha_1\kappa_2^3+\alpha_2 \frac{\om^4}{\kappa_2}\right)\right).
\]

We assume now that that $\lambda/\mu \gg 1$. For the unperturbed problem, the properties of the
generalized eigenvalues are given in Table~\ref{tab:1},
\[
\tilde s_0^2\simeq -0.9,\quad \frac{\kappa_1}{|\om|} \simeq 0.3,\quad \frac{\kappa_2}{|\om|} \simeq 1.
\]
Therefore,
\begin{alignat*}{2}
\theta(\tilde{s}_0^2) \simeq &\, \frac{\lambda h^2}{2 \mu |\om|}\frac{\kappa_2}{|\om|}\left(
\alpha_1\frac{\kappa_1^3}{|\om|^3}|\om|^3 + \alpha_2\om^2\frac{\kappa_1}{|\om|}|\om| - 0.55\left(
\alpha_1\frac{\kappa_2^3}{|\om|^3}|\om|^3 + \alpha_2 \om^2 |\om|
\frac{|\om|}{\kappa_2}\right)\right)\\ 
\simeq &\, \frac{\lambda h^2\om^2}{2 \mu}\bigl(0.027\alpha_1 +
0.3\alpha_2-0.55(\alpha_1+\alpha_2)\bigr).
\end{alignat*}

We want to estimate how sensitive the generalized eigenvalues $s_0=\pm i\xi_0$ are to truncation
error perturbations. We perturb
$\varphi(\tilde{s})$ around $\tilde{s}_0$. For $\lambda/\mu \gg 1$, Table~\ref{tab:1} and \eqref{varphi-prime} gives
\begin{alignat*}{2}
\frac{\kappa_1}{|\om|}\frac{\kappa_2}{|\om|} \frac{\varphi'(\tilde{s}_0)}{\tilde{s}_0} &\simeq
\frac{\kappa_2^2}{\om^2} - 2\left(1+\frac{\tilde{s}_0^2}{2}\right)^3 \simeq 0.67
\\
\varphi'(\tilde{s}_0) &\simeq \tilde{s}_0 \frac{0.67}{0.3} \simeq \pm 2.12\, i.
\end{alignat*}
The Taylor expansion \eqref{taylor-gen-eig} gives for small $h\om$, 
\[
(\tilde{s} - \tilde{s}_0) \varphi'(\tilde{s}_0) = \theta(\tilde{s}_0).
\]
We get
\begin{equation}\label{perturbed-gen-eig}
\tilde{s} - \tilde{s}_0 \simeq \frac{\theta(\tilde{s}_0^2)}{\varphi'(\tilde{s}_0)} 
\simeq \mp \frac{\lambda h^2\om^2}{2\mu}\bigl(0.027\alpha_1 +
0.3\alpha_2-0.55(\alpha_1+\alpha_2)\bigr)
\frac{i}{2.12}.
\end{equation}

We now make some observations.  Because $\theta(\tilde{s}_0^2)$ is real, the generalized eigenvalue
is perturbed along the imaginary axis and remains purely imaginary. Hence the perturbed problem is
well-posed. The value of the perturbed generalized eigenvalue determines the phase velocity of
surface waves in the numerical solution. To avoid large phase errors, we must therefore keep the
perturbation of the generalized eigenvalue small. If we accept a relative error in the phase speed of size
$\epsilon$, where $0<\epsilon\ll 1$, we have to choose the grid size $h$ such that
\begin{equation}\label{gen-eig-accuracy}
\frac{\lambda h^2\om^2|\alpha_0|}{\mu} = \epsilon \lsls 1.
\end{equation}
If the computational grid has $P$ grid points per wave length $L=2\pi/|\om|$, we get
\[
h = \frac{L}{P} = \frac{2\pi}{|\om| P},\quad h|\om| = \frac{2\pi}{P},\quad P = 2\pi
\left(\frac{|\alpha_0|}{\epsilon} \frac{\lambda}{\mu}\right)^{1/2}.
\]
Hence, the number of grid points per wave length must be proportional to $\sqrt{\lambda/\mu}$ to
maintain the accuracy as $\mu/\lambda\to 0$.

For a fourth order accurate method, where the leading order truncation error terms are
\[
g_1=\alpha'_1h^4\frac{\p^5 u}{\p x^5}+\alpha'_2h^4 \frac{\p^5 v}{\p y^5},
\]
equation \eqref{gen-eig-accuracy} is replaced by
\begin{equation}\label{gen-eig-4}
\frac{\lambda h^4\om^4|\alpha_0'|}{\mu} = \epsilon \lsls 1.
\end{equation}
The number of grid points per wave length to maintain an $\epsilon$-error in the phase velocity now
becomes
\[
P = 2\pi
\left(\frac{|\alpha_0'|\lambda}{\epsilon\mu}\right)^{1/4}.
\]
Therefore, as $\mu/\lambda\to 0$, the number of grid points per wave length grows much slower for
the 4th than the 2nd order accurate method.

For other truncation error perturbations of the boundary conditions, such as a $u_{xxxx}$ term in
$g_1$, $\theta(\tilde{s}_0^2)$ becomes complex. If the truncation error coefficient has the wrong
sign, the perturbed problem gets eigenvalues with positive real part. From Lemma~\ref{lem:necessary}
we know that such problems are ill-posed. Furthermore, the factor $\mu$ in the denominator of
\eqref{perturbed-gen-eig} shows that the rate of the exponential growth can get arbitrarily large as
$\mu/\lambda\to 0$. It is therefore very difficult to compensate for such growth with an artificial
dissipation term.

\section{Numerical experiments}\label{sec:num}

For a second order hyperbolic equation, energy conservation ensures that all eigenvalues of the
spatial operator are either real and negative, or zero. The same property applies to the discretized
problem. To avoid any spurious growth in the numerical solutions, it is therefore important to use a
discretization that also satisfies energy conservation. Such a discretization was derived for the
3-D elastic wave equation in Nilsson et al~\cite{NilPetSjoKre-07}. In the present work, we use the
corresponding discretization for the two-dimensional case. This numerical method discretizes the
elastic wave equation with a second order accurate, energy conserving, finite difference method on a
Cartesian grid with constant grid sizes in space and time. The second order method was recently
generalized to fourth order accuracy by Sjogreen and Petersson~\cite{PetSjo-4th}, and we use both
the second and forth order methods in the following numerical experiments. Note that our finite
difference methods are based on solving the elastic wave equation as a second order hyperbolic
system using summation by parts operators. These methods are fundamentally different from the commonly
used staggered grid method developed by Vireaux~\cite{virieux-86}, Levander~\cite{levander-88}, and
others, which is based on solving the elastic wave equation as a first order hyperbolic system.

\subsection{Surface waves}\label{sec:s-waves}
To study surface waves using real arithmetic, we are interested in the real part of the eigenfunction
\eqref{gen-efcn} corresponding to the generalized eigenvalue
\[
s = i\xi_0.
\]
Assuming $\omega > 0$, the real part of \eqref{gen-efcn} can be written as
\begin{multline}\label{rayleigh}
\ub_s(x,y,t) = e^{-\omega\sqrt{1-\tilde{\xi}_0^2}\,x}\begin{pmatrix}
\cos\bigl(\omega( y + c_r t)\bigr)\\ 
\sqrt{1-\tilde{\xi}_0^2}\sin\bigl(\omega( y + c_r t)\bigr)
\end{pmatrix} \\
+ \left( \frac{\tilde{\xi}_0^2}{2} - 1\right)e^{-\omega\sqrt{1 - \tilde{\xi}_0^2\mu/(2\mu+\lambda)}\, x}\begin{pmatrix}
\cos\bigl(\omega( y + c_r t)\bigr)\\ 
\sin\bigl(\omega( y + c_r t)\bigr)/\sqrt{1 - \tilde{\xi}_0^2\mu/(2\mu+\lambda)}
\end{pmatrix}.
\end{multline}
Here, we define the Rayleigh phase velocity by
\[
c_r = \tilde{\xi}_0\sqrt{\mu}.
\]

To perform reliable numerical simulations, it is of great interest to know the number of grid points
per wave length, $P$, that is required to obtain a certain accuracy in a numerical solution. If the
wave length is $L=2\pi/|\om|$, we define
\[
P=\frac{L}{h}.
\]

We consider a periodic domain in the $y$-direction and choose the computational domain to contain
exactly one wave length of the solution. In this investigation we shall keep the wave length fixed
at $L=1$, which gives the spatial frequency $\omega=2\pi$. For simplicity, we set $\lambda=1$ in all
numerical experiments. A free surface boundary condition is imposed at $x=0$. We truncate the
computational domain at $x=L_x$ where we impose an inhomogeneous Dirichlet condition. The boundary
data is given by the exact solution \eqref{rayleigh}, which is exponentially small along $x=L_x$.
For all values of $\lambda/\mu$ (see Table~\ref{tab:1}), $\kappa_1/|\om| > 0.2955$, and we make the
influence of the Dirichlet boundary closure small by choosing
\[
L_x=10,\quad e^{-2\pi\sqrt{1-\tilde{\xi}_0^2}L_x} \leq e^{-2\pi\cdot 0.2955 \cdot 10} \approx
  8.5\cdot 10^{-9}.
\]

In our first experiment, we take $\mu=0.01$. The numerical solution is evolved from
initial data given by \eqref{rayleigh} at time $t=0$ and $t=-\delta_t$, where the time step
satisfies the Courant condition (recall that we have scaled time to give unit density)
\[
\delta_t = K_{C}\frac{h}{\sqrt{\lambda+3\mu}},\quad K_C=\begin{cases}
0.9,&\mbox{second order method},\\
1.3,&\mbox{fourth order method}.
\end{cases}
\]
In Figure~\ref{fig:rayleigh}
we show the max norm of the error in the numerical solution as function of time for $t\leq 20$.
\begin{figure}
  \begin{center}
    \includegraphics[height=0.41\textwidth]{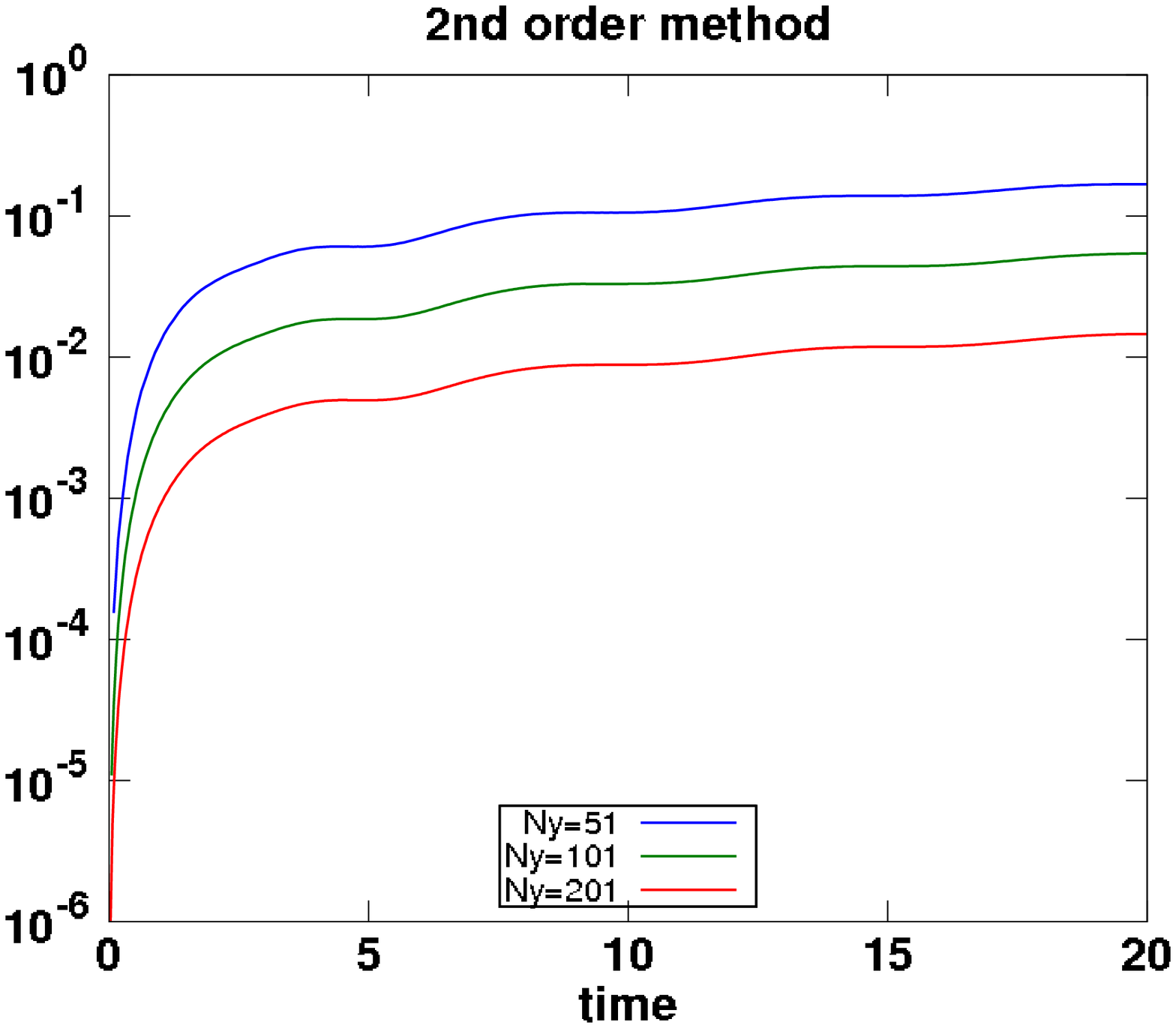}
    \includegraphics[height=0.41\textwidth]{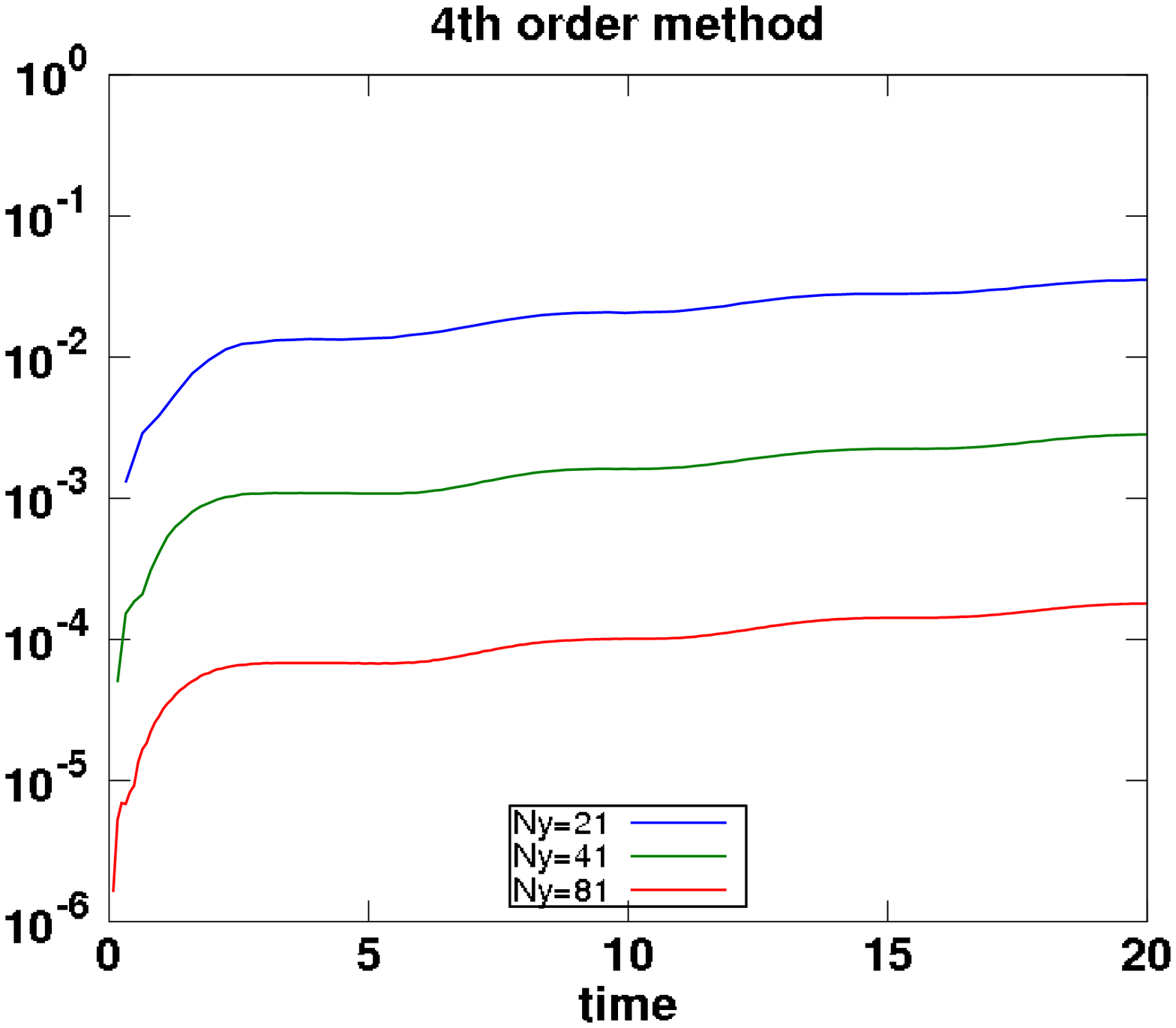}
    \caption{Max error as function of time for a Rayleigh surface wave with $\lambda=1$ and
      $\mu=0.01$. Results from the second and fourth order methods are shown on the left and right,
      respectively. The different colors corresponds to different number of grid points per wave
      length. Note that the grids are coarser for the fourth order computations.}
    \label{fig:rayleigh}
  \end{center}
\end{figure}
Since the wave length in the $y$-direction is one, the number of grid points per wave length
satisfies $P=N_y-1$. Results for the second order accurate method are shown on the left,
illustrating the expected convergence rate as the grid is refined. Note that at least 100 grid points
per wave length (green line) are needed to obtain a numerical solution to within about 5\% of the exact
solution. On the right side of the same figure, we show results for the fourth order
method. Here the error decreases by a factor of 16 when the number of grid points is doubled. In
this case, only 20 grid points per wave length are needed to make the error less than about 5\% of
the exact solution.

In our next experiment, we study how the accuracy depends on $\mu$ when the second order method is
used for propagating the Rayleigh wave \eqref{rayleigh}. The period of the wave is
\begin{equation}\label{rayleigh-period}
T = \frac{2\pi}{\om c_r} =\frac{1}{c_r}=\frac{1}{\tilde{\xi}_0 \sqrt{\mu}}.
\end{equation}
In Table~\ref{tab:num3} we show the max norm of the error after one and ten periods.  Note that the
period gets longer, i.e., the surface wave propagates slower as $\mu\to 0$. The case $\mu=0.1$ shows
close to second order convergence, both at time $t=T$ and $t=10\, T$. The error levels are
reasonable for a second order method, but increase with time because the error is dominated by phase
errors, i.e., the numerical solution propagates with a slightly different phase velocity compared to
the exact solution.  The error gets larger for $\mu=0.01$, and a finer grid must used to obtain
comparable error levels. For $\mu=0.001$, the grid must be refined further to obtain reasonable
error levels, and the cases $P=25$ and $P=50$ are inadequate. A visual inspection shows that after
10 periods, the numerical solution with $P=50$ is more than $180^\circ$ out of phase with the exact
solution (experiment not shown to save space). We only observe close to second order convergence
when the grid is refined from 200 to 400 grid points per wave length.
\begin{table}
\begin{center}\small
\begin{tabular}{l|c|c|c}
Case      & $P=1/h$ & $\| u_{err} \|_\infty(t=T)$ & $\| u_{err} \|_\infty(t=10\, T)$ \\ \hline
$\mu=0.1$ & 25     & $5.26\cdot 10^{-2}$         & $2.99\cdot 10^{-1}$ \\ 
$T=3.330$ & 50     & $1.48\cdot 10^{-2}$         & $9.26\cdot 10^{-2}$ \\ 
          & 100    & $3.85\cdot 10^{-3}$         & $2.44\cdot 10^{-2}$ \\ \hline
$\mu=0.01$ & 25     & $2.45\cdot 10^{-1}$         & $7.64\cdot 10^{-1}$ \\ 
$T=10.474$ & 50     & $1.07\cdot 10^{-1}$         & $5.59\cdot 10^{-1}$ \\ 
           & 100    & $3.32\cdot 10^{-2}$         & $2.06\cdot 10^{-1}$ \\ 
           & 200    & $8.86\cdot 10^{-3}$         & $5.73\cdot 10^{-2}$ \\ \hline
%
%$\mu=0.001$ & 25    & $4.53\cdot 10^{-1}$         & $5.07\cdot 10^{-1}$ \\ 
%$\mu=0.001$ & 50    & $3.49\cdot 10^{-1}$         & $6.97\cdot 10^{-1}$ \\ 
$\mu=0.001$ & 100   & $1.96\cdot 10^{-1}$         & $7.69\cdot 10^{-1}$ \\ 
$T=33.104$  & 200   & $7.40\cdot 10^{-2}$         & $4.25\cdot 10^{-1}$ \\ 
            & 400   & $2.13\cdot 10^{-2}$         & $1.36\cdot 10^{-1}$ \\ 
\end{tabular}
\end{center}
\caption{The max norm of the error in the numerical evolution of the Rayleigh surface wave, after
  one and ten periods. Note how the number of grid points per wave length, $P=1/h$, must be
  drastically increased to maintain the accuracy as $\mu$ becomes smaller. }\label{tab:num3}
\end{table}

Note that for $\mu=0.001$, the grid with 200 grid points per wave length gives of the order 10
percent accuracy after one period ($\| \ub_s \|_\infty\approx 0.545$). This grid is about 10 times
finer than what is normally required to get that accuracy with a second order
method~\cite{KreOli72}. In the $x$-direction, the gradient of the exact solution is the largest
along $x=0$, and $|v_x|=|u_y|$ for all $\mu$. In the limit $\mu\to 0$, it is straight forward to
show $|u_x|=|v_y|$. Hence, the gradient of the exact solution is of the same order in both
directions, and conclude that solution is extremely well resolved on the grid.  Furthermore, the phase
velocity of the surface wave becomes slower and slower as $\mu\to 0$, while the time step is
governed by $\sqrt{\lambda+3\mu}$, which tends to $\sqrt{\lambda}=1$. Hence, the temporal resolution
of the surface wave only improves as $\mu\to 0$.

The analysis of the phase velocity in \S\ref{sec:trunc} shows that truncation errors in a second
order accurate method perturb the generalized eigenvalue according to
\begin{equation}\label{xi-pert}
\tilde{\xi} = \tilde{\xi}_0 + \epsilon,\quad \epsilon = \frac{\lambda h^2\om^2|\alpha_0|}{\mu}.
\end{equation}
The perturbed generalized eigenvalue corresponds to a perturbed phase velocity $c_r' =
\sqrt{\mu}\tilde{\xi}$. Assuming that phase errors dominate the numerical errors, the amplitude of
the error follows by
\[
e(t) = \om(c_r' - c_r)t = \om\sqrt{\mu}\,\epsilon t.
\]
The period of the surface wave follows from \eqref{rayleigh-period}, so $e(T)= C_1\epsilon T$, $C_1=\mbox{const}$. For a
computational grid with grid size $h=1/P$, \eqref{xi-pert} gives
\[
\epsilon = \frac{C_2 \lambda}{ P^2 \mu},\quad C_2=\mbox{const}.
\]
Hence, to maintain a constant error level in the numerical solution after a fixed number of periods,
we must choose $P\sqrt{\mu}=\mbox{const.}$, if $\lambda$ is constant. This assertion is tested by
the numerical experiment shown on the left side of Figure~\ref{fig:4}. Here we show the max error as
function of time scaled by the period of the solution. The first case (red curve) corresponds to
$\mu=0.1$, with period $T=3.33$ and resolution $P=40$ grid points per wave length. Notice how
closely this error curve follows the case $\mu=10^{-3}$, with period $T=33.104$ and a grid with 400
grid point per wave length. We conclude that the second order method needs a prohibitively fine
computational grid to accurately calculate surface waves for small values of $\mu$. 
\begin{figure}
\begin{centering}
\includegraphics[width=0.47\textwidth]{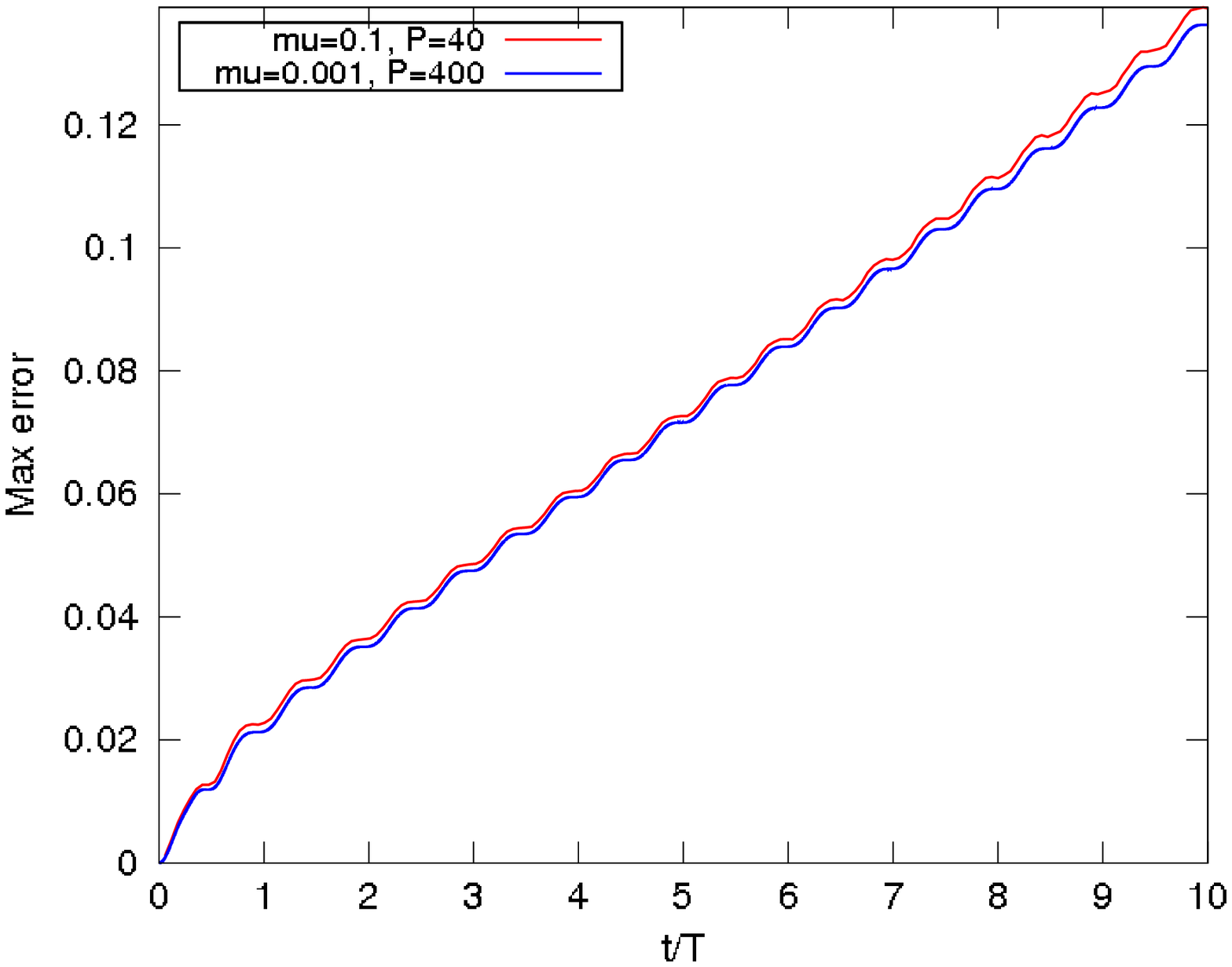}
\includegraphics[width=0.47\textwidth]{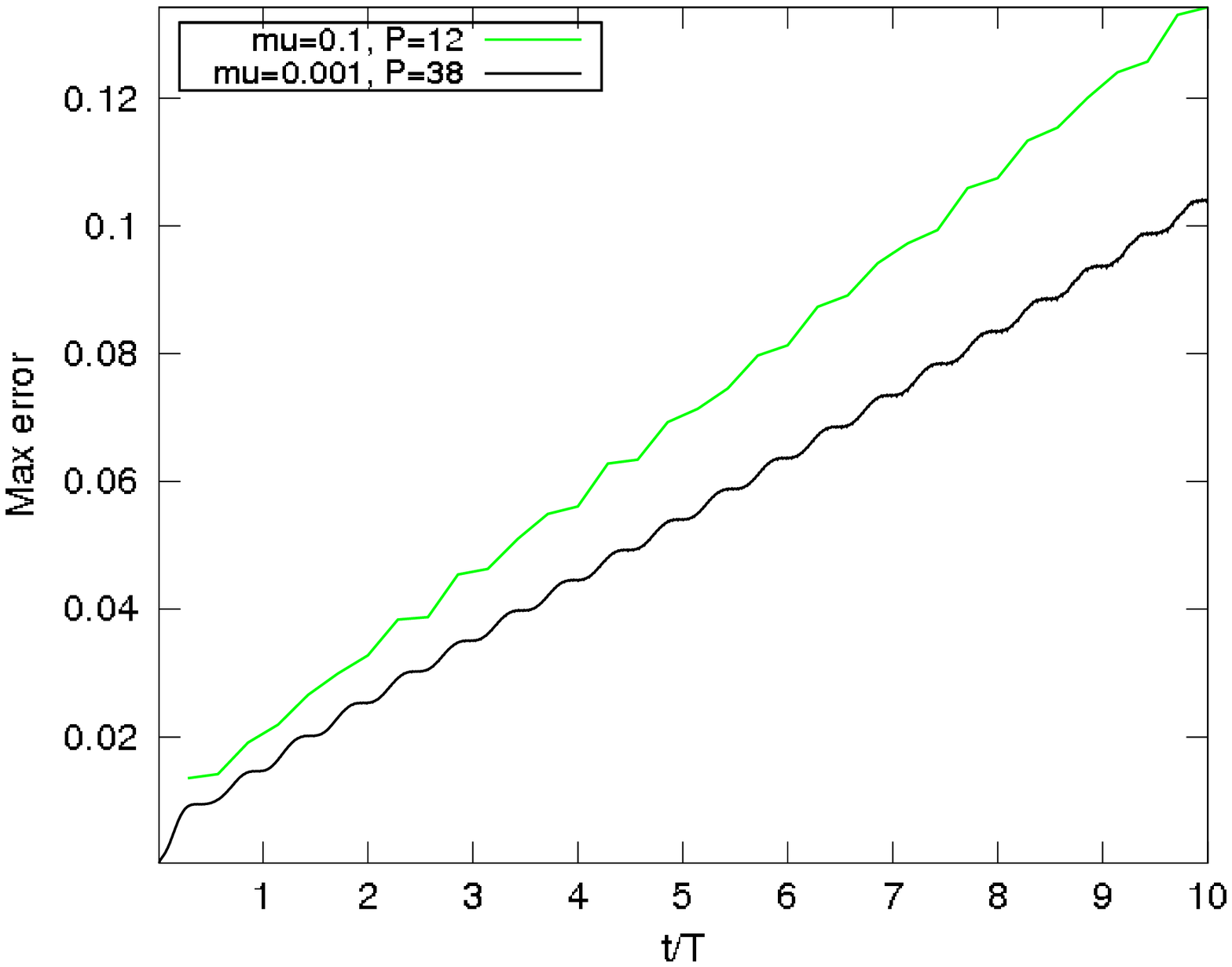}
\caption{Max error in the numerical evolution of the Rayleigh surface wave, as function of time
  scaled by the period, $T=1/(\tilde{\xi}_0\sqrt{\mu})$. For the second order method (left), the
  case $\mu=0.1$ with $P=40$ is shown in red and $\mu=0.001$ with $P=400$ is shown in blue. For the
  fourth order method (right), the case $\mu=0.1$ with $P=12$ (green) and $\mu=0.001$ with $P=38$
  (black) give comparable error levels.}
\label{fig:4}
\end{centering}
\end{figure}

We repeat the above experiment with a fourth order accurate method. The results are shown on the
right side of Figure~\ref{fig:4}. In this case we obtain similar error levels using a
significantly coarser grid. For $\mu=0.1$ and $\mu=0.001$, we use $P=12$ and $P=38$,
respectively. For the fourth order method, the perturbation of the generalized eigenvalue is given
by \eqref{gen-eig-4}. Using the same argument as for the second order method, we must choose
$P\mu^{1/4}=\mbox{const.}$ to obtain a constant error level in the numerical solution after a fixed
number of periods. This scaling is approximately preserved in these calculations, since
\[
P \mu^{1/4} \approx \begin{cases}
6.748,&P=12,\ \mu=0.1,\\
6.757,&P=38,\ \mu=0.001.
\end{cases}
\]
We conclude that the fourth order method is much better suited for simulations when $\mu$ is
small. Compared to the second order method, the fourth order method needs a smaller number of grid
points per wave length, and the required resolution grows much slower as $\mu\to 0$. 

To indicate how much more efficient the fourth order method is in practice, we give some execution
times obtained on a MacBook Pro laptop computer. The above numerical experiments for $\mu=0.001$
required $20,604$ seconds ($\approx 5$ hours, 43 minutes) for the second order method with
$P=400$. Similar accuracy was obtained with the fourth order method using $P=38$, but this
calculation only took $60$ seconds. Hence, for this problem the fourth order method was 343 times
faster than the second order method.

\subsection{Mode to mode conversion}

Consider a compressional wave of unit amplitude traveling in the negative $x$-direction in a
homogeneous material, with displacement
\[
\ub^{(in)} = \begin{pmatrix} k \\ \om \end{pmatrix} e^{i(\xi t + k x + \om y)},\quad
k=\cos\phi>0,\quad \om = \sin\phi,\quad \xi>0.
\]
If this wave encounters a free surface boundary at $x=0$, it will be reflected and split into two
waves that both travel in the positive $x$-direction,
\begin{alignat*}{2}
\ub^{(out)} &= \ub^{(P}) + \ub^{(S)},\\
\ub^{(P)} &= R_p \begin{pmatrix} -k \\ \om \end{pmatrix} e^{i(\xi t - k x + \om y)},\\
\ub^{(S)} &= \frac{R_s}{\sqrt{\alpha^2 k^2 + \om^2}} \begin{pmatrix} -\om \\ -\alpha k \end{pmatrix} e^{i(\xi t -
  \alpha k x + \om y)},\quad \alpha>0.
\end{alignat*}
The reflected waves correspond to a compressional and a shear wave, since the curl of $\ub^{(P)}$
and the divergence of $\ub^{(S)}$ are zero. In order for $\ub^{(in)}$ and $\ub^{(out)}$ to satisfy
the elastic wave equation \eqref{uv-eqn} with $F_1=F_2=0$, the frequency and wave numbers must
satisfy the elementary relations
\begin{equation}\label{dispersion}
\xi^2 = (\lambda+2\mu)(k^2 + \om^2) = \lambda+2\mu,\quad \xi^2 = \mu(\alpha^2k^2 + \om^2).
\end{equation}
We select the signs of $\xi$ and $\alpha$ such that $\ub^{(in)}$ and $\ub^{(out)}$ travel in
the negative and positive $x$-direction, respectively. The amplitudes of the reflected waves, $R_p$
and $R_s$, are functions of $\lambda$, $\mu$, and the angle of the incident wave, $\phi$. The
amplitudes $R_p$ and $R_s$ are uniquely determined by the free surface boundary conditions
\eqref{uv-bc} (with $g_1=g_2=0$). For a more detailed discussion, we refer to
Achenbach~\cite{achenbach}, \S~5.6.

As a consequence of the relation \eqref{dispersion},
\[
\alpha^2 = 1+\frac{\lambda+\mu}{\mu\cos^2\phi}.
\]
Hence, when $\mu\ll \lambda$, the reflected S-wave will propagate almost parallel to the $x$-direction because
$\alpha^2\gg 1$, see Figure~\ref{fig:6}. The wave lengths of the compressional and shear waves are
given by
\[
L_p = \frac{2\pi}{\sqrt{k^2+\om^2}} = 2\pi,\quad L_s = 2\pi \sqrt{\frac{\mu}{\lambda+2\mu}}.
\]
Note that the wave length of the compressional wave is fixed, while $L_s$ becomes small as $\mu\to 0$. 
\begin{figure}
\begin{centering}
\includegraphics[width=0.32\textwidth]{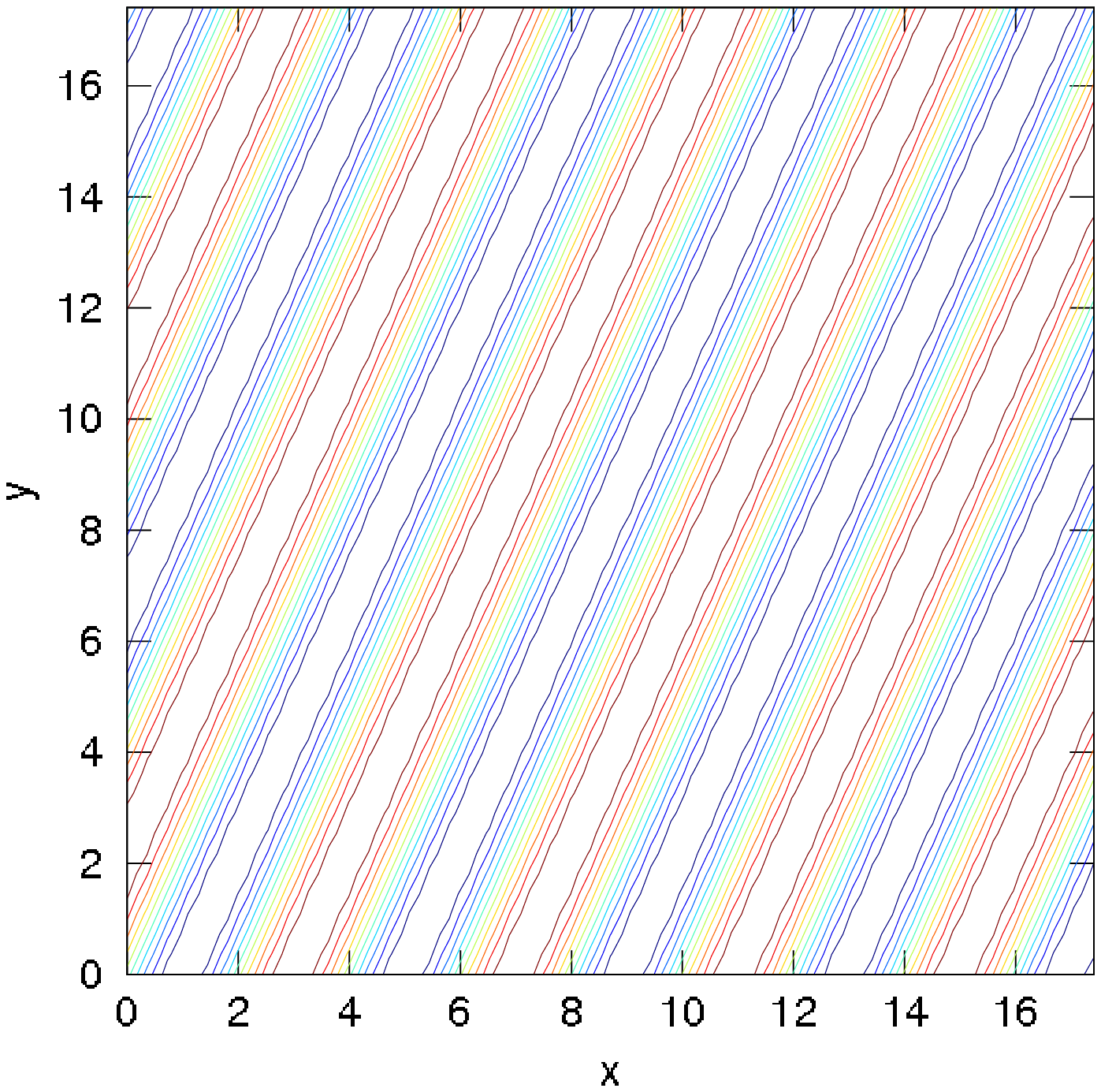}
\includegraphics[width=0.32\textwidth]{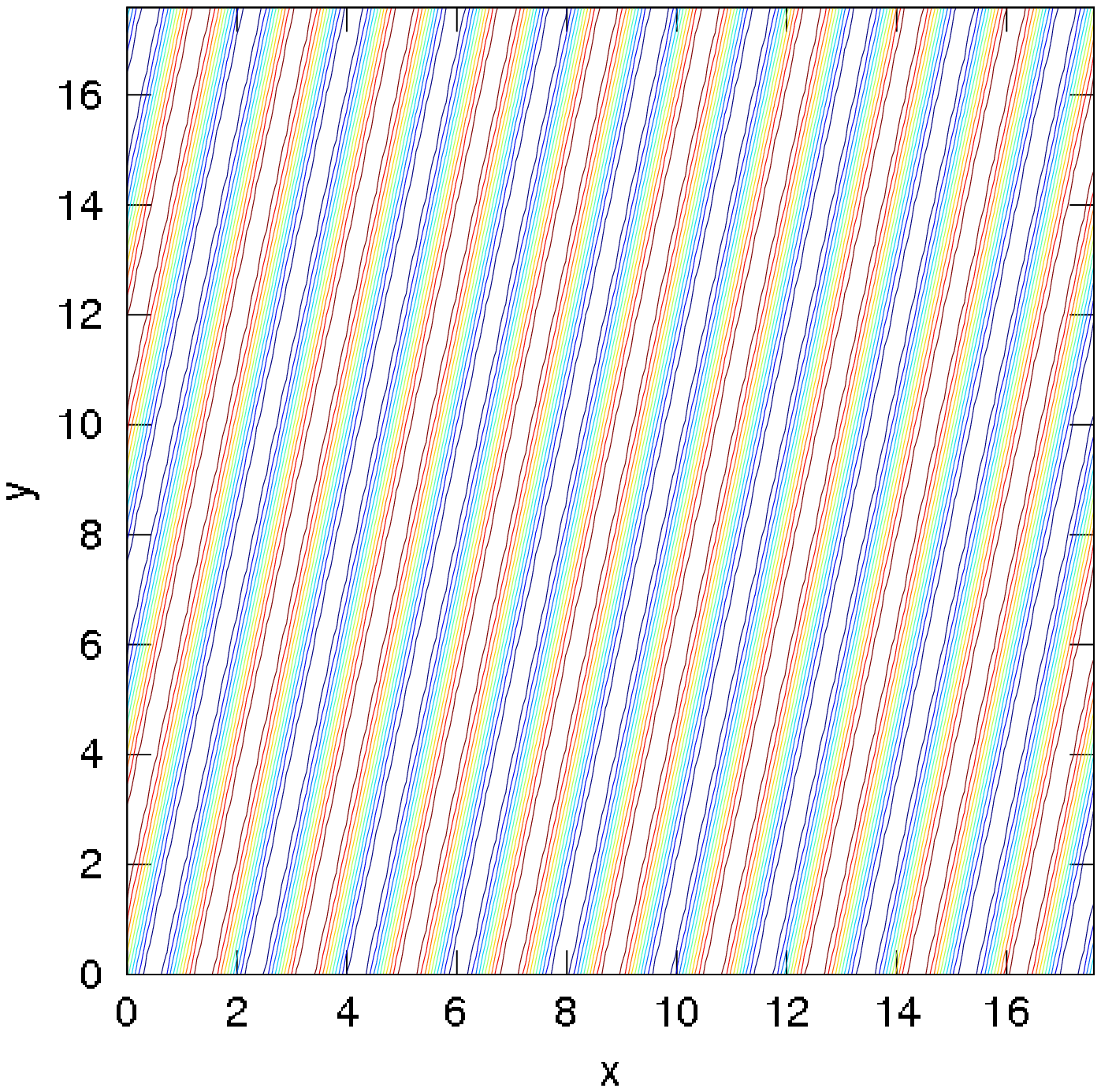}
\includegraphics[width=0.32\textwidth]{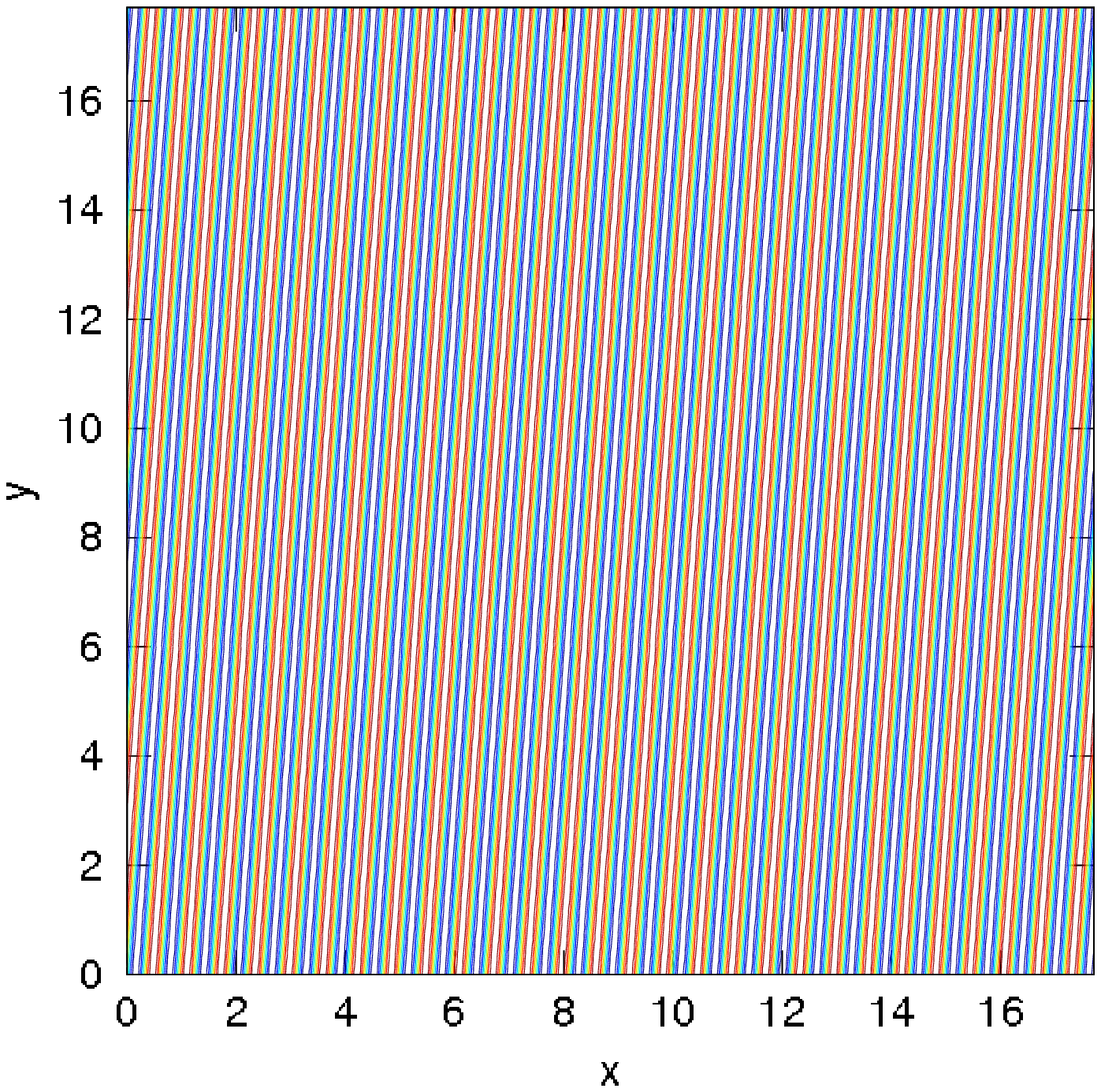}
\caption{$v$-component of the outgoing shear wave as function of $(x,y)$ at $t=0$. The angle of the
  incoming P-wave is $\phi=\pi/4$.  The frames correspond to $\mu=1.0$ (left), $\mu=0.1$ (middle),
  and $\mu=0.01$ (right).}
\label{fig:6}
\end{centering}
\end{figure}

To include two wave lengths of $\ub^{(in)}$ in the computational domain, we take $L_y=4\pi/\sin\phi$ and
$L_x=4\pi/\cos\phi$. As before, we impose periodic boundary conditions in the $y$-direction, a
Dirichlet boundary condition at $x=L_x$ and a free surface condition at $x=0$. By construction, the
function $\ub^{(in)} + \ub^{(out)}$ is $L_y$-periodic in the $y$-direction, satisfies the elastic
wave equation in the interior, and the free surface condition at $x=0$. In principle, we could compute a numerical
approximation of $\ub^{(in)} + \ub^{(out)}$ by adding a suitable forcing function to the Dirichlet
boundary condition at $x=L_x$. However, we instead choose to only compute the outgoing S-wave,
$\ub^{(S)}$. For this reason, we impose the inhomogeneous Dirichlet boundary condition
\[
\ub(L_x,y,t) = \ub^{(S)}(L_x,y,t),
\]
and take the forcing functions in the normal stress boundary conditions
\eqref{uv-bc} to be
\begin{alignat*}{2}
g_1 &= -\left( u^{(in)}_x + u^{(P)}_x\right) - \gamma^2 \left( v^{(in)}_y + v^{(P)}_y\right),\\
g_2 &= -\left( u^{(in)}_y + u^{(P)}_y + v^{(in)}_x + v^{(P)}_x\right).
\end{alignat*}
We use the exact solution $\ub^{(S)}$ as initial conditions for the numerical solution.

To accurately solve this problem numerically, it is necessary to resolve the short shear waves on
the computational grid. For this problem, we define the resolution in terms of the number of grid
points per shear wave length,
%\begin{equation}\label{shear-res}
\[
P_s = \frac{L_s}{h} = \frac{\sqrt{\mu}}{h} \frac{2\pi}{\sqrt{\lambda+2\mu}}.
\]%\end{equation}
We evaluate the error in the numerical solution as function of time for two materials. The first
material has ($\lambda=1$, $\mu=0.1$) and the second has ($\lambda=1$, $\mu=0.01$). As a
consequence, the period of the wave is slightly different for the two cases
\[
T = \frac{2\pi}{\xi} = \frac{2\pi}{\sqrt{\lambda+2\mu}}\approx
\begin{cases}
5.74, & \mu=0.1, \\
6.22, & \mu=0.01.
\end{cases}
\]

In Figure~\ref{fig:5} we show the error as function of normalized time, $t/T$, for the two
materials, using the fourth order accurate method.
\begin{figure}
\begin{centering}
\includegraphics[width=0.6\textwidth]{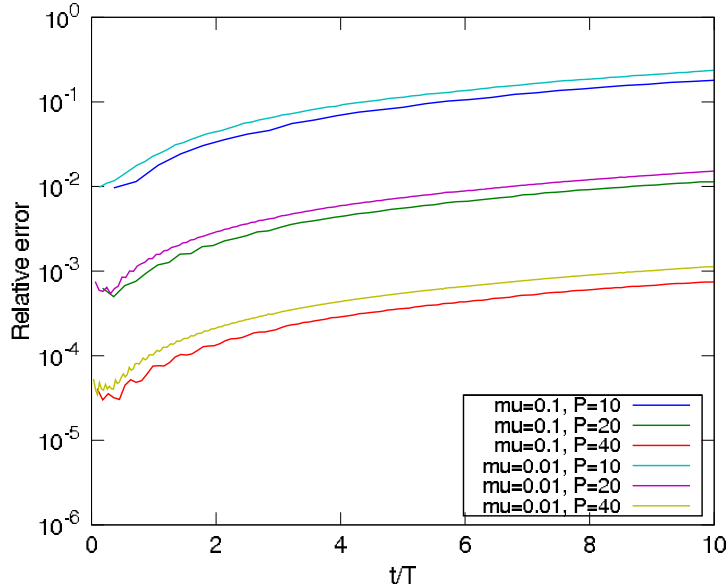}
\caption{ Results for computing the outgoing shear wave with different resolution, characterized by
  the number of grid points per wave length, $P_s$. The relative error in max norm is shown as
  function of time scaled by the period of the wave. Two cases are shown, ($\lambda=1$, $\mu=0.1$)
  and ($\lambda=1$, $\mu=0.01$).}
\label{fig:5}
\end{centering}
\end{figure}
Note that the error levels are comparable for the same number of grid points per wave length, and
converge to zero as ${\cal O}(P_s^{-4})$ as the grid is refined. Thus the mode to mode conversion
problem does {\em not} suffer from the same extreme resolution requirements as the surface wave
problem in the previous section. Because we have scaled the problem such that the P-waves have wave
length $2\pi$, the S-waves get a wave length of the order $2\pi\sqrt{\mu}$. Hence, to keep the number of grid
points per S-wave length constant for different materials, we have to choose the grid size according
to
\[
h = \frac{2\pi}{\sqrt{\lambda+2\mu}}\frac{\sqrt{\mu}}{P_s}.
\]
Compared to the material with $\mu=0.1$, the grid size must therefore be taken about a factor of $\sqrt{10}$
smaller for the case $\mu=0.01$, to obtain the same number of grid points per wave length. This
scaling is independent of the order of accuracy in the numerical method.  

No surface waves can be triggered by a propagating P-wave because the relation \eqref{dispersion}
shows that $\xi^2/(\mu \om^2)>1$. However, evanescent modes due to an interior forcing function could trigger
both S-waves and surface waves. Since the surface waves are only slightly slower than the S-waves,
their wave length is of the same order as the length of an S-wave of the same frequency. If the
problem is scaled such that the P-wave length is constant, both the S-wave and the surface waves would
therefore have wave lengths of the order $\sqrt{\mu}$. Based on the results of Section~\ref{sec:s-waves}, a
second order accurate method would need a grid size of the order $h\sim \mu$ to maintain a constant
accuracy in the numerical solution as $\mu\to 0$. For a fourth order method, it would suffice to use $h\sim \mu^{3/4}$.
 
%% in the mode to mode conversion problem
%% Hence,
%% the total number of grid points in a two-dimensional calculation grows as ${\cal O}(\mu^{-1})$ as
%% $\mu\to 0$, while the number of time steps per period grows as ${\cal O}(\mu^{-1/2})$. Note that
%% this asymptotic scaling is valid for any numerical method that requires a fixed number of grid points per wave
%% length. In particular, the exponents are independent of the order of accuracy.

\section{Conclusions}\label{sec:conc}
We have developed a normal mode analysis for the half-plane problem of the elastic wave equation
subject to a free surface boundary condition. Our analysis allows the solution to be estimated in
terms of the boundary data, showing that the solution is as smooth as the boundary forcing. Hence,
using the terminology of~\cite{Kreiss-Ortiz-Petersson}, the problem is {\it boundary stable}. The
dependence on the material properties is transparent in our estimates. Using a modified equation
approach, the normal mode technique was extended to analyze the influence of truncation errors in a
finite difference approximation. Our analysis explains why the number of grid points per wave length
must be so large when calculating surface waves in materials with $\mu/\lambda\ll 1$. To obtain a
fixed error in the phase velocity of Rayleigh surface waves, our analysis predicts that the grid
size must be proportional to $\mu^{1/2}$ for a second order method, when $\lambda=\mbox{const}$. For
a fourth order method, the analysis shows that it suffices to use $h\sim\mu^{1/4}$. These scalings
have been confirmed by numerical experiments.

It is theoretically possible to derive stable finite difference schemes that give higher than fourth
order accuracy. These methods use wider stencils that are more expensive to evaluate, but for the
surface wave problem, it would suffice to use a grid size of the order $h\sim\mu^{1/p}$, where $p$
is the order of accuracy. For sufficiently small values of $\mu$ these methods should be
more efficient as the order of accuracy increases. However, numerical experiments must be performed
to evaluate how small $\mu$ must actually be to compensate for the higher computational
complexity of these very high order accurate methods.

\section{Acknowledgments}
We thank Tom Hagstrom for discussions that lead to a simple proof of Lemma~\ref{lem:2-anders}.

\appendix
\section{Miscellaneous lemmata}\label{app:a}
\begin{lemma}\label{lem:2-anders}
Let $\om$ be a real number and let $s=\eta+i\xi$ be a complex number where $\eta>0$. Consider the relation
\[
\kappa=\sqrt{\om^2+s^2},
\]
where the branch cut in the square root is defined by
\begin{equation}\label{branch-cut}
 -\pi< \arg(\alpha+i\beta)\le \pi,\quad \arg\sqrt{\alpha+i\beta}=\halv\arg(\alpha+i\beta).
\end{equation}
Then,
\begin{equation}
 \Rk \ge \eta.
\end{equation}
\end{lemma}
\begin{proof}
Since $\eta>0$, we can write
\begin{equation}\label{kappa-def}
s = \eta(1+i{\xi'}),\quad \kappa = \eta\sqrt{{\om'}^2 + (1+i{\xi'})^2},\quad
{\xi'}=\frac{\xi}{\eta},\quad {\om'} =\frac{\om}{\eta}.
\end{equation}
Define real numbers $a$ and $b$ such that
\begin{equation}\label{a+ib}
a + i b = \sqrt{\tilde{\om}^2 + (1+i{\xi'})^2},\quad a\geq 0.
\end{equation}
Squaring relation \eqref{a+ib} and identifying the real and imaginary parts give
\begin{alignat*}{2}
ab &= {\xi'},\\
a^2 - b^2 &=  1 - {\xi'}^2 + {\om'}^2.
\end{alignat*}
The first relation gives $b={\xi'}/a$, which inserted into the second relation results in
\begin{equation}\label{f-a2}
a^2 - \frac{{\xi'}^2}{a^2} = 1 - {\xi'}^2 + {\om'}^2.
%,\quad f(a^2) =: a^2 - \frac{{\xi'}^2}{a^2}.
\end{equation}
Note that the left hand side is a monotonically increasing function of $a^2$. When ${\om'}=0$, 
equation \eqref{f-a2} is solved by $a^2=1$. The right hand side of \eqref{f-a2} is a monotonically increasing
function of ${\om'}^2$. Therefore $a^2>1$ for ${\om'}^2>0$. We conclude that the unique
solution of \eqref{f-a2} satisfies
\[
a^2\geq 1.
\]
Because $a$ must be non-negative, we have $a\geq 1$. Relations \eqref{kappa-def} and \eqref{a+ib}
give
\[
\Rk = \eta\, {\rm Re}\,\sqrt{{\om'}^2 + (1+i{\xi'})^2} = \eta\, a \geq \eta.
\]
\end{proof}

\begin{corollary}
Let $\om$ be a real number, $s=\eta + i\xi$ be a complex number with $\eta>0$, and let $0<\gamma <\infty$ be a constant. Then
there is another constant $0<\delta<\infty$ such that
\begin{equation}
{\rm Re}\,\sqrt{\om^2+\frac{s^2}{\gamma^2}} \ge \delta\eta,\quad \eta = {\rm Re}\,(s) > 0.
\end{equation}
\end{corollary}
\begin{proof}
Let ${s'}=s/\gamma$. Lemma \ref{lem:2-anders} proves that
\[
{\rm Re}\,\sqrt{\om^2+\frac{s^2}{\gamma^2}} = {\rm Re}\,\sqrt{\om^2+{s'}^2} \geq {\rm
  Re}\,({s'}) = \frac{1}{\gamma}{\rm Re}\,(s).
\]
Hence, $\delta=1/\gamma>0$ and the corollary follows.
\end{proof}

\section{The case $\om\to 0$}\label{sec:omega->0}
We now extend the boundary estimates in Section~\ref{sec:b-force} to the case $\om\to 0$ when
$s=s_0$, $|s_0|>0$ is fixed. In this limit,
\[
\left|\tilde{s}^2\right| = \left|\frac{s^2}{\mu\om^2}\right| \to \infty.
\]
For $|\tilde{s}|\gg 1$, we can simplify \eqref{eq:21} according to
\begin{equation}\label{eq1-mu0}
\tilde{s}^2\sqrt{\frac{\mu}{\lambda+2\mu}} \, \hat{u}_{01} + \frac{\tilde{s}^2}{2}\, \hat{u}_{02} =
-\frac{\lambda + 2\mu}{2\mu|\om|}\sqrt{\frac{\mu}{\lambda+2\mu}}\,\tilde{s}\hat{g}_1.
\end{equation}
In a similar way, \eqref{eq:22} becomes
\begin{equation}\label{eq2-mu0}
\hat{u}_{01} + \frac{2}{\tilde{s}^2}\, \hat{u}_{02} = -\frac{i}{2\om \tilde{s}^2}\, \hat{g}_2.
\end{equation}
Solving the latter equation for $\hat{u}_{01}$ and inserting into \eqref{eq1-mu0} gives
\[
\left( -2\sqrt{\frac{\mu}{\lambda+2\mu}} + \frac{\tilde{s}^2}{2}\right) \, \hat{u}_{02} =
-\frac{(\lambda + 2\mu)\tilde{s}\,\hat{g}_1}{2\mu|\om|}\sqrt{\frac{\mu}{\lambda+2\mu}} +
\frac{i\, \hat{g}_2}{2\om} \sqrt{\frac{\mu}{\lambda+2\mu}}
\]
For large $|\tilde{s}|$, we have to leading order,
\begin{equation}\label{eq3-mu0}
\hat{u}_{02} = -\sqrt{\frac{\lambda+2\mu}{\mu}}\frac{\hat{g}_1}{|\om|\tilde{s}} +
\sqrt{\frac{\mu}{\lambda+2\mu}}\,\frac{i\,\hat{g}_2}{\om \tilde{s}^2}
\end{equation}
Note that $s=\tilde{s}\sqrt{\mu}|\om|$, and
\[
\frac{1}{\tilde{s}|\om|} = \frac{\sqrt{\mu}}{s},\quad \frac{1}{\tilde{s}^2|\om|} =
\frac{\mu|\om|}{s^2},\quad
\frac{1}{\tilde{s}^2} = \frac{\mu \om^2}{s^2}.
\]
Hence, \eqref{eq3-mu0} and \eqref{eq2-mu0} give
\[
\hat{u}_{02} = -\sqrt{\lambda+2\mu}\,\frac{\hat{g}_1}{s} + {\cal O}(\om),\quad
\hat{u}_{01} = -\frac{i\mu\om\,\hat{g}_2}{2s^2} + {\cal O}(\om^2),\quad \om\to 0.
\]
The solution on the boundary follows from \eqref{uv-bndry} and gives directly
\[
\hat{u}(0) = \hat{u}_{01} + \hat{u}_{02} = -\frac{\sqrt{\lambda+2\mu}\,\hat{g}_1}{s} + {\cal O}(\om).
\]
For large $|\tilde{s}|$, we can simplify the expression for $\hat{v}(0)$,
\[
\hat{v}(0) = -\frac{i|\om|\tilde{s}\,\hat{u}_{01}}{\om} -
\frac{i\om\,\hat{u}_{02}}{|\om|\tilde{s}}\sqrt{\frac{\lambda+2\mu}{\mu}} =
-\frac{\sqrt{\mu}\,\hat{g}_2}{2s} + {\cal O}(\om).
\]
The expressions for $\hat{u}(0)$ and $\hat{v}(0)$ show that the solution is well behaved in the
limit $\om\to 0$.

\bibliography{refs}

\end{document}